\documentclass[a4paper]{amsart}
\usepackage{etex}
\usepackage{enumerate}
\usepackage[english]{babel}
\usepackage[T1]{fontenc}
\usepackage{calligra}
\usepackage{enumerate}
\usepackage{amscd}
\usepackage{amsmath}
\usepackage{color}
\usepackage{mathrsfs}
\usepackage{amsfonts}
\usepackage[mathcal]{euscript}
\usepackage{amssymb}
\usepackage{color}
\usepackage{amsthm}

\theoremstyle{definition}
\newtheorem{definition}{Definition}[section]
\newtheorem{example}[definition]{Example}
\newtheorem{remark}[definition]{Remark}

\theoremstyle{plain}
\newtheorem{theorem}[definition]{Theorem}

\newtheorem{lemma}[definition]{Lemma}

\numberwithin{equation}{section}

\def \R {\mathbb{R}}

\def \loc {\mathrm{loc}}
\def \N {\mathbb{N}}
\def \c {\mathbf{c}}
\def \KK {\mathcal{K}}

\def \HH {\mathcal{H}}
\def \BB {\mathcal{B}}

\def \TT {\mathcal{T}}
\def \DD {\mathcal{D}}

\def \GG {{\mathcal{G}}}
\def \AA {{\mathcal{A}}}
\def \d {\mathrm{d}}
\def \de {\partial}

\def \MM {\mathrm{M}}

\title[Singular BVPs with functional terms]{Boundary value problems associated with \\
 singular strongly nonlinear equations \\ with functional terms}

\author[S. Biagi]{Stefano Biagi}
\address{Stefano Biagi\hfill\break\indent
Dipartimento di Matematica \hfill\break\indent
Politecnico di Milano \hfill\break\indent
Via Bonardi, 9 \hfill\break\indent
20133 Milano (Italy)}
\email{stefano.biagi@polimi.it}

\author[A. Calamai]{Alessandro Calamai}
\address{Alessandro Calamai\hfill\break\indent
Dipartimento di Ingegneria Civile, Edile e Architettura \hfill\break\indent
Universit\`a Politecnica delle Marche\hfill\break\indent
Via Brecce Bianche, 12\hfill\break\indent
60131 Ancona (Italy)}
\email{calamai@dipmat.univpm.it}

\author[C. Marcelli]{Cristina Marcelli}
\address{Cristina Marcelli \hfill\break\indent
Dipartimento di Ingegneria Industriale e Scienze Matematiche \hfill\break\indent
Universit\`a Politecnica delle Marche\hfill\break\indent
Via Brecce Bianche, 12\hfill\break\indent
60131 Ancona (Italy)}
\email{marcelli@dipmat.univpm.it}

\author[F. Papalini]{Francesca Papalini}
\address{Francesca Papalini \hfill\break\indent
Dipartimento di Ingegneria Industriale e Scienze Matematiche \hfill\break\indent
Universit\`a Politecnica delle Marche\hfill\break\indent
Via Brecce Bianche, 12\hfill\break\indent
60131 Ancona (Italy)}
\email{papalini@dipmat.univpm.it}


\date{}

\begin{document}  
\begin{abstract} 
We study boundary value problems associated with singular,
strongly nonlinear differential equations with functional terms of type
$$\big(\Phi(k(t)\,x'(t))\big)' + f(t,\GG_x(t))\,\rho(t, x'(t)) = 0$$
on a compact interval $[a,b]$.
These equations are quite general due to the presence of 
a strictly increasing homeomorphism $\Phi$, the so-called $\Phi$-Laplacian operator,
of a nonnegative function $k$, which may vanish on a set of null measure,
and moreover of a functional term $\GG_x$.
We look for solutions, in a suitable weak sense, which belong to the Sobolev space $W^{1,1}([a,b])$.
Under the assumptions of the existence of a well-ordered pair of upper and lower solutions and
of a suitable Nagumo-type growth condition, we prove an existence result by means of 
fixed point arguments.
\end{abstract}  
  \maketitle

  \section{Introduction} \label{sec.intro}
  The main aim of this paper is to study the solvability
  (in a suitable weak sense) of boundary value problems (BVPs, in short)
  of the following form:
  \begin{equation} \label{eq.BVPintro}
   \begin{cases}
     \big(\Phi(k(t)\,x'(t))\big)' + f(t,\GG_x(t))\,\rho(t, x'(t)) = 0 & \text{a.e.\,on
     $I := [a,b]$}, \\[0.1cm]
     x(a) = \HH_a[x],\,\,x(b) = \HH_b[x].
   \end{cases}   
  \end{equation}
where $\Phi:\R\to\R$, the so-called \emph{\(\Phi\)-Laplacian operator}, is a strictly increasing 
ho\-meo\-mor\-phism,
$k:I\to \R$ is a bounded nonnegative function satisfying 
$$1/k\in L^1(I),$$
$f$ and $\rho$ are Carath\'{e}odory functions, and $\GG_x$,
$\HH_a$, $\HH_b$ are \emph{functional} terms, i.e.,
\begin{itemize}
     \item $\GG:W^{1,1}(I)\to L^\infty(I)$ is a continuous operator which 
      verifies suitable bound\-ed\-ness and monotonicity conditions;
      \item $\HH_a,\,\HH_b:W^{1,1}(I)\to\R$ are continuous
      and increasing operators.
\end{itemize}
In particular, we stress that the differential equation
\begin{equation} \label{eq.intro}
    \big(\Phi(k(t)\,x'(t))\big)' + f(t,\GG_x(t))\,\rho(t, x'(t)) = 0 \quad \text{a.e.\,on } I
\end{equation}
is quite general, since it contains a functional term which can be 
\emph{non-local} or \emph{delayed} (see examples in Section \ref{sec.examples}).
Moreover, due to the fact that
the function $k$ may vanish on a set having zero Lebesgue measure,
 as a particular case of \eqref{eq.intro} one gets the \emph{singular} ODE
\[
    \big(\Phi(k(t)\,x'(t))\big)' + f(t,x(t))\,\rho(t, x'(t)) = 0 \quad \text{a.e.\,on } I
\] 
which has been already studied by the authors in the recent papers
\cite{BCP_heteroclinek, CaMaPa2018}.
In the context of singular ODEs, it seems natural to look for solutions with weak regularity, namely in 
$W^{1,1}(I)$ rather than in $C^{1}(I)$.
As a consequence, 
when con\-si\-de\-ring the BVP \eqref{eq.BVPintro} we assume that all the 
involved operators are defined in the Sobolev space $W^{1,1}(I)$.

The \(\Phi\)-Laplacian operator can be considered as a generalization of the classical $p$-Laplace operator $\Phi_p(z) := |z|^{p-2}z$.
The study of different BVPs, both on boun\-ded and on unbounded domains, associated 
with equations with \(\Phi\)-Laplacian, like
\begin{equation} \label{philap.intro}
(\Phi(x'))'=f(t,x,x'),
\end{equation}
is motivated by applications, e.g.\ in non-Newtonian
fluid theory, diffusion of flows in porous media, nonlinear elasticity
and theory of capillary surfaces.
Many au\-thors have considered such kind of problems and proposed generalizations in various directions;
see, e.g., \cite{BeJeMa10, BeMa07, BeMa08, CabORPou,CabPou,CabPou2,KFA}.
We also refer the reader to the survey \cite{Cab2011} and to the references therein.
Let us mention here equations with mixed differential operators, that is,
\begin{equation} \label{philap.atx.intro}
\left(a(t,x)\Phi(x')\right)' =  f(t,x,x'),
\end{equation}
where $a$ is a continuous positive function 
(see, e.g., \cite{Biagi_ANPA, Biagi_Isernia, Ca, CupMarPap}). In the au\-to\-no\-mous case, namely, 
$$a(t,x)\equiv a(x),$$
equation \eqref{philap.atx.intro} also arises in some models, 
e.g.\  reaction-diffusion equations with non-constant diffusivity and
porous media equations.

In this framework, a typical approach to get existence results
is given by the combination of 
fixed point techniques 
and the method of upper and lower solutions.
A crucial tool which gives a priori bounds for the derivatives of the solutions is a Nagumo-type growth condition on the nonlinearity.
Recently, in the paper \cite{MaPa17} the authors obtained an existence result assuming a weak form of Wintner-Nagumo growth condition. The approach of \cite{MaPa17} has been fruitfully extended to the context of singular equations: see  \cite{Biagi_ANPA, BCMP_first, 
BCP_heteroclinek, Biagi_Isernia, CaMaPa2018}.
In our main result (see Theorem \ref{thm.mainPAPER} below) we assume the following weak 
Nagumo growth condition:
\begin{equation*}
 | f(t,z)\rho(t,y)|\le \psi(|\Phi(k(t)y)|)\cdot\left( \ell(t)+\mu(t)|y|^{\frac{q-1}{q}}\right), \label{ip:psi2}
 \end{equation*}
where \(\mu \in L^q(I)\), $q>1$, \(\ell\in  L^1(I)\), \(\psi\) is measurable and such that
\[ \int_1^{+\infty}
\frac{\d s}{\psi(s)} = +\infty. \]
This assumption allows to consider a very general operator $\Phi$.

As pointed out, in this paper we turn our attention on singular equations with functional terms, 
both inside the differential equation and in the boundary conditions. 
As far as we know, equations involving the \(\Phi\)-Laplacian operator and functional terms are
less studied and understood due to technical difficulties,
see  \cite{AORS06, HS05}.
Singular equations with \(\Phi\)-Laplacian are also few studied, and just for a 
restricted class of nonlinearities
(see \cite{Liu2012,LY2015}).
Thus, the coexistence of all these features (singular equations with \(\Phi\)-Laplacian and functional terms) makes our problem particularly challenging from a theoretical point of view.
\medskip

While we refer to Section \ref{sec.examples} for some concrete examples
illustrating the ap\-pli\-ca\-bi\-li\-ty of our results, here we limit ourselves
to point out that our approach allows us to prove the solvability of, e.g.,
$$\begin{cases}
         \big(\Phi_p
         \big(|\sin(t)|^{1/\vartheta_0}\,x'(t)\big)\big)' + 
         x_{\tau}(t)\,|x'(t)|^\delta = 0
         & \text{a.e.\,on $I$}, \\[0.25cm]
         \,\,x(0) = \sqrt[3]{x(\pi)},\,\,
         x(2\pi) = \frac{1}{4\pi}\int_0^{2\pi}({x(s)}+{2})\,\d s
        \end{cases}$$
where $\Phi_p(z) = |z|^{p-2}z$ is the usual $p$-Laplace operator, $\vartheta_0,\delta$
are positive constants 
and the functional term $\GG_x = x_{\tau}$ is of delay-type, that is,
$$x_\tau(t) := \begin{cases}
x(t-\tau), & \text{for $t\in[0,2\pi],\,t\geq\tau$},\\
x(0), & \text{otherwise}
\end{cases}$$
\noindent A brief plan of the paper is now in order. 
\begin{itemize}
 \item[$\diamond$] In Section \ref{sec.mainresult} we fix some preliminary definitions
 and we state our main existence result, namely Theorem \ref{thm.mainPAPER}.
 
 \item[$\diamond$] In Section \ref{sec.proofTHM} we provide
 the proof Theorem \ref{thm.mainPAPER}, which articulates
 into two steps: 
 first, we perform a truncation argument and we introduce an auxiliary BVP
 to which suitable existence results do apply; 
 then, we show that \emph{any} solution the `truncated' problem
 is actually a solution of the original BVP. 
  In doing this, we use in a crucial way the assumption of the existence 
 of a well-ordered pair of lower and upper solutions of our problem. 
 
 \item[$\diamond$] In Section \ref{sec.examples} we present some examples to which 
 our Theorem \ref{thm.mainPAPER} applies. 
 
 \item[$\diamond$] Finally, we close the paper with an Appendix 
 containing the explicit proof of a technical Lemma,
which in some previous papers was missing, and in other papers was either not complete or not correct.
\end{itemize}

   \section{Preliminaries and main results} \label{sec.mainresult}
     Let $a,b\in\R$ satisfy $a<b$, and let $I:=[a,b]$. As mentioned
     in the 
     In\-tro\-duc\-tion,
     throughout this paper we shall be concerned
     with BVPs of the following form
     \begin{equation} \label{eq.BVPmainSec2}
      \begin{cases}
       \big(\Phi(k(t)\,x'(t))\big)' + f(t,\GG_x(t))\,\rho(t, x'(t)) = 0  & \text{a.e.\,on
       $I$}, \\[0.1cm]
       x(a) = \HH_a[x],\,\,x(b) = \HH_b[x],
      \end{cases}   
     \end{equation}
     where  \(\Phi\) is a strictly increasing homeomorphism, 
     \(f, \rho: I\times \R\to \R\) are Carath\'eodory functions, and $k,\,\GG, \, \HH_a,\,\HH_b$
     satisfy the following assumptions:
     \begin{itemize}
      \item[(H1)] $k:I\to\R$ is a \emph{nonnegative} function satisfying
      \begin{equation} \label{eq.assumptionkconcrete}
		k\in L^\infty(I)\qquad\text{and}\qquad 1/k\in L^1(I).
      \end{equation}       

      \item[(H2)] $\GG:W^{1,1}(I)\to L^\infty(I)$ is continuous 
      (with respect to the usual norms) and
       \emph{bounded
       when $W^{1,1}(I)$ is thought of as a subspace} of $L^\infty(I)$; this means,
       precisely, that
       for every $r > 0$ there exists $\eta_r > 0$ such that
       \begin{equation} \label{eq.assumptionGGbounded}
      \|\GG_x\|_{L^\infty(I)}\leq \eta_r \quad
      \text{for any $x\in W^{1,1}(I)$ with $\|x\|_{L^\infty(I)}\leq r$}.
       \end{equation}
       
       \item[(H3)] There exists a constant $\kappa \geq 0$ such that
       \begin{equation} \label{eq.assumptionGGmonotone}
        \begin{array}{c}
      \text{$f(t,\GG_x(t))+\kappa\,x(t)
      \leq f(t,\GG_y(t))+\kappa\,y(t)$ a.e.\,on $I$} \\[0.2cm]
      \text{for every $x,y\in W^{1,1}(I)$ such that
      $x\leq y$  a.e. on $I$}.
      \end{array}
       \end{equation}
      
      \item[(H4)] $\HH_a,\,\HH_b:W^{1,1}(I)\to\R$ are continuous
      (with respect to the usual to\-po\-lo\-gies)
      and monotone increasing, that is, 
      \begin{equation} \label{eq.assumptionHHaHHbmonotone}
      \begin{array}{c}
      \text{$\HH_a[x]\leq\HH_b[y]$\quad and\quad $\HH_b[x]\leq \HH_b[y]$} \\[0.2cm]
      \text{for every $x,y\in W^{1,1}(I)$ such that
      $x\leq y$ a.e. on $I$}.
      \end{array}
      \end{equation}
     \end{itemize}
     \begin{remark} \label{rem.continuityGG}
      We point out, for a future reference, that
      the continuity of $\GG$ from $W^{1,1}(I)$ into
      $L^\infty(I)$ is ensured if
      $\GG$ maps continuously
      $W^{1,1}(I)$ into some Banach space
      $(X,\|\cdot\|_X)$ which is \emph{continuously embedded}
      into $L^\infty(I)$. \vspace*{0.05cm}
      
      This is the case, e.g.,
      of the following functional spaces:
      \begin{itemize}
       \item[(1)] $X = W^{1,p}(I)$ (with $p\geq 1$ and the usual norm);       
       
       \item[(2)] $X = C^n(I,\R)$ (with $n\in\N$ and the usual norm).
      \end{itemize}
      Moreover, we also notice that the monotonicity
      assumption (H3) seems very na\-tu\-ral to get existence results for problem \eqref{eq.BVPintro}.
      We point out that a similar mo\-no\-to\-ni\-city assumption has already 
      been considered by the authors in a 
      different context in the paper \cite{CaMaPa2009}.       
      \end{remark}
     \begin{remark} \label{rem.assumptionkconcrete}
      We explicitly notice that, in view of assumption (H1), the function
    $k$ can vanish on a set $E\subseteq\R$ of zero Lebesgue measure
    (in particular, $E$ could be infinite). As a consequence,
    the ODE appearing in \eqref{eq.BVPmainSec2}
      may be singular.
     \end{remark}
    
	 The aim in this paper is
      to study the solvability of
     \eqref{eq.BVPmainSec2} in a weak sense, according to the following definition. 
     \begin{definition} \label{def.solutionconcrete}
      We say that a function $x\in W^{1,1}(I)$ is a \emph{solution}
      of problem \eqref{eq.BVPmainSec2} if it satisfies
      the following two properties:
      \begin{itemize}
     \item[(1)] the map $t\mapsto \Phi(k(t)\,x'(t))$ is in $W^{1,1}(I)$ and
     $$\big(\Phi(k(t)\,x'(t))\big)' 
     + f(t,\GG_x(t))\,\rho(t, x'(t)) = 0 \qquad\text{for a.e.\,$t\in I$};$$
     \item[(2)] $x(a) = \HH_a[x]$ and $x(b) = \HH_b[x]$.
    \end{itemize}
    If $x$ satisfies only property (1), we say that $x$ is a \emph{solution
    of the ODE}
    \begin{equation} \label{eq.ODEconcrete}
    \big(\Phi(k(t)\,x'(t))\big)' + f(t,\GG_x(t))\,\rho(t, x'(t)) = 0.
    \end{equation} 
     \end{definition}
     Another fundamental notion
     for our investigation of
     the solvability of
     \eqref{eq.BVPmainSec2} is 
     the notion of lower/upper solution,
     which is contained in the next definition.
     \begin{definition} \label{def.lowerupper}
      We say that a function $x\in W^{1,1}(I)$ is a
      \emph{lower 
      \emph{[resp.\,}upper\emph{]} 
      solution} of problem \eqref{eq.BVPmainSec2} if it satisfies
      the following two properties:
      \begin{itemize}
     \item[(1)] the map $t\mapsto \Phi(k(t)\,x'(t))$ is in $W^{1,1}(I)$ and
     $$\big(\Phi(k(t)\,x'(t))\big)' +
     f(t,\GG_{x(t)})\,\rho(t, x'(t))
     \geq\,[\leq]\,\,\,0 \qquad\text{for a.e.\,$t\in I$};$$
     \item[(2)] $x(a) \leq\,[\geq]\,\,\,\HH_a[x]$ and 
     $x(b) \leq\,[\geq]\,\,\,\HH_b[x]$.
    \end{itemize}
    If the function
    $x$ satisfies only property (1), we say that it is a 
    \emph{lower 
      \emph{[resp.\,}upper\emph{]} 
      solution of the ODE} in \eqref{eq.BVPmainSec2}.

     \end{definition}
     \begin{remark} \label{rem.functionKx}
       If $u\in W^{1,1}(I)$ is any function such that
       $$t\mapsto\Phi(k(t)\,u'(t))\in W^{1,1}(I)$$
       (this is the case, e.g., of any lower/upper
       solution of \eqref{eq.ODEconcrete}), 
       the continuity of $\Phi^{-1}$ implies the existence
       of a (unique) continuous function  $\mathcal{K}_u$ such that
       $$\text{$\mathcal{K}_u(t) = k(t)\,u'(t)$ for a.e.\,$t\in I$
       \quad and \quad $\Phi\circ\KK_u\in W^{1,1}(I)$}.$$
    In particular, this is true if $u$ is a \emph{solution}
    of \eqref{eq.ODEconcrete}.
   \end{remark}
     After all these preliminaries, we can state the main result of the paper.

     \begin{theorem} \label{thm.mainPAPER}
      Let the structural assumptions \emph{(H1)}-to-\emph{(H4)} be in force. More\-over,
      let us suppose that the following 
      additional hypotheses are satisfied:
      \begin{itemize}
       \item[\emph{(H5)}] there exist a 
       lower solution $\alpha$ 
       and an upper solution $\beta$ of problem \eqref{eq.BVPmainSec2}
       which are
       \emph{well-ordered} on $I$, that is, $\alpha(t)\leq\beta(t)$ for every $t\in I$;
       \item[\emph{(H6)}] for every $R > 0$ and every non-negative
       function $\gamma\in L^1(I)$ there exists a non-negative
       function $h = h_{R,\gamma}\in L^1(I)$ such that
       \begin{equation} \label{eq.growthfrhoass}
        \begin{array}{c}
        |f(t,z)\,\rho(t, y(t))|\leq h_{R,\gamma}(t) \\[0.2cm]
        \text{for a.e.\,$t\in I$, every $z\in\R$ with $|z|\leq R$} \\[0.05cm]
        \text{and every
        $y\in L^1(I)$ such that $|y(s)|\leq \gamma(s)$ for a.e.\,$s\in I$}.
        \end{array}
       \end{equation}
       \item[\emph{(H7)}] for every $R > 0$ there exist
       a constant $H = H_R > 0$, a non-negative fun\-ction $\mu = \mu_R\in L^q(I)$
       \emph{(}with $1<q\leq\infty$\emph{)}, a non-negative
       function $l = l_R\in L^1(I)$ and a measurable function
       $\psi =\psi_R:(0,\infty)\to(0,\infty)$ such that
       \begin{align}
	 &(\ast)\,\,1/\psi\in L^1_{\loc}(0,\infty)
	 \quad \text{and} \quad 
	 \int^\infty \frac{1}{\psi(t)}\,\d t = \infty; \label{eq.integralpsidiv} \\[0.2cm] 
	 &(\ast\ast)\begin{array}{l}
	 |f(t,z)\,\rho(t, y)| \leq \psi\big(|\Phi(k(t)y)|\big)
	 \cdot\big(l(t)+\mu(t)\,|y|^{\frac{q-1}{q}}\big); \\[0.15cm]
	 \text{for a.e.\,$t\in I$, any $z\in [-R,R]$ and any $y\in\R$ with
	$|k(t) y| \geq H$.}
	\end{array} \label{eq.Nagumocondition}
	\end{align}
      \end{itemize}
      Then, there exists a
     solution $x_0\in W^{1,1}(I)$ of 
    problem \eqref{eq.BVPmainSec2},
	fur\-ther sa\-ti\-sfyi\-ng 
	\begin{equation} \label{eq.x0traalphabeta}
	 \alpha(t)\leq x_0(t)\leq\beta(t) \qquad\text{for every $t\in I$}.
	\end{equation}
	Moreover, the following higher-regularity properties
	hold:
	\begin{itemize}
	 \item[\emph{(1)}] if
	$1/k\in L^\vartheta(I)$ for some $1<\vartheta\leq\infty$, 
	one also has that $x_0\in W^{1,\vartheta}(I)$;
	 \item[\emph{(2)}] if 
	 $k\in C(I,\R)$ and $k > 0$ on $I$, one also has that $x_0\in C^1(I,\R)$.
	\end{itemize}
	Finally, if $M > 0$ is any real number such that
	$\|\KK_\alpha\|_{L^\infty(I)},\,\|\KK_\beta\|_{L^\infty(I)}\leq M$, there
	exists a constant $L_M > 0$, only depending on $M$, such that
	\begin{equation} \label{eq.boundKKx0LM}
	 \|x_0\|_{L^\infty(I)}\leq M \qquad\text{and}\qquad\|\KK_{x_0}\|_{L^\infty(I)}\leq L_M.
	\end{equation}
     \end{theorem}      
     \begin{remark} \label{rem.existencealphabeta}
      Despite its relevance in our argument, the existence of a 
      well\--or\-de\-red pair of lower and upper solutions
     $\alpha,\beta$ for \eqref{eq.BVPmainSec2} is not obvious
     (see, e.g., \cite{Cab2011, MaPa14} and the
     reference therein for general results on this topic).
     Here, we limit ourselves to observe that, if $\rho(t,0) = 0$
     for every $t\in I$,
     then \emph{any constant function}
     is both a lower and an upper solution \emph{for the ODE}
     \eqref{eq.ODEconcrete}. \vspace*{0.05cm}
     
     As a positive
     counterpart of the previous comment, we shall present in
     the next Section \ref{sec.examples}
     a couple of examples of BVPs to which
     Theorem \ref{thm.mainPAPER} applies.
     \end{remark}
     \section{Proof of Theorem \ref{thm.mainPAPER}} \label{sec.proofTHM}
     The proof of Theorem \ref{thm.mainPAPER} is rather
     technical and long; for this reason, after
     ha\-ving introduced some constants and parameters used throughout,
     we shall pro\-ceed by establishing several claims.
     Roughly put, our approach consists of two steps. \medskip
     
     \textsc{Step I}: As a first step, by crucially exploiting the existence
     of a well-ordered pair of lower and upper solutions
     $\alpha,\beta$ for \eqref{eq.BVPmainSec2} (see, precisely, assumption (H5)), 
     we perform a truncation argument and we introduce
     a new problem, say $\mathrm{(P)_\tau}$, to which some
     abstract results do apply. \medskip
     
     \textsc{Step II}: Then, we show that \emph{any} solution
     of $\mathrm{(P)_\tau}$ is actually a solution
     of \eqref{eq.BVPmainSec2}. In doing this, we use again
     in a crucial way the fact that $\alpha$ and $\beta$ are,
     re\-spec\-ti\-ve\-ly, a lower and an upper solution for \eqref{eq.BVPmainSec2}.  
     \medskip
     
     \noindent We then begin by fixing some quantities which shall be
     used all over the proof. \bigskip
     
     First of all,
      we choose a real $M > 0$ in such a way that
      $\|\alpha\|_{L^\infty(I)}\leq M$ and $\|\beta\|_{L^\infty(I)}\leq M$.
      Moreover, using assumption (H2), we let
      $\eta_M > 0$ be such that
      \begin{equation} \label{eq.choiceetaM}
       \|\GG_u\|_{L^\infty(I)}\leq \eta_M \qquad\text{for all
       $u\in W^{1,1}(I)$ with $\|u\|_{L^\infty(I)}\leq M$}.
      \end{equation}  
      With reference to assumption (H7), we then set
      $$H_M := H_{\eta_M}, \qquad \mu_M := \mu_{\eta_M}, \qquad
      l_M = l_{\eta_M}, \qquad \psi_M := \psi_{\eta_M}.$$
      Now, since
      $\Phi$ is strictly increasing, we can choose $N > 0$ such that
      \begin{equation} \label{eq.choiceN}
       \begin{split}
       & \Phi(N) > 0,\,\,\,\Phi(-N) < 0 \qquad\text{and} \\[0.15cm]
       & \quad N > \max\Big\{H_M,\frac{2M}{b-a}\cdot\|k\|_{L^\infty(I)}\Big\};
       \end{split}
      \end{equation}
      accordingly, owing to \eqref{eq.integralpsidiv}, we fix $L = L_M \geq N > 0$ in such a way that
      \begin{equation} \label{eq.choiceL}
      \begin{split}
      & \min\bigg\{\int_{\Phi(N)}^{\Phi(L_M)}\frac{1}{\psi_M}\,\d s,  
       \int^{-\Phi(-L_M)}_{-\Phi(-N)}\frac{1}{\psi_M}\,\d s\bigg\} \\[0.15cm]
       & \qquad\qquad > \|l_M\|_{L^1(I)}+\|\mu_M\|_{L^q(I)}\cdot(2M)^{\frac{q-1}{q}},
       \end{split}
      \end{equation}
	  and we consider the function $\gamma_L\in L^1(I)$ defined as: 
      \begin{equation} \label{eq.defigammaL}
       \gamma_L(t) := \frac{L_M}{k(t)}+|\alpha'(t)|+|\beta'(t)|.
      \end{equation}
     Following the notation in Appendix \ref{app.continuityTT}, we also define
     the truncating operators
     \begin{equation} \label{eq.defiOpTT}
      \TT := \TT^{\,\alpha,\beta} \qquad\text{and}\qquad
      \DD := \TT^{\,-\gamma,\gamma}.
     \end{equation}
     Given any $x\in W^{1,1}(I)$, we then consider the function $F_x$ defined by
     \begin{equation} \label{eq.defiOpFtrunc}
     \begin{array}{c}
      F_x(t)  :=
      -f\big(t,\GG_{\TT_x}(t)\big)\,\rho\big(t, \DD_{\TT_x'}(t)\big)+
      \arctan\big(x(t)-\TT_x(t)\big).
      \end{array}
     \end{equation}
     Finally, we consider the
      operators $\BB_a,\,\BB_b:W^{1,1}(I)\to\R$ defined as
     \begin{equation} \label{eq.defiBBtrunc}
      \BB_a := \HH_a\circ \TT, \qquad \BB_b := \HH_b\circ\TT.
	\end{equation}            
      Thanks to all these preliminaries, we can finally introduce
     the following BVP (which can be thought of as a
     truncated version of problem \eqref{eq.BVPmainSec2}):
     \begin{equation} \label{eq.BVPtrunc}
      \begin{cases} 
      \big(\Phi(k(t)\,x'(t))\big)' = F_x(t)  & \text{a.e.\,on
       $I$}, \\[0.1cm]
       x(a) = \BB_a[x],\,\,x(b) = \BB_b[x].
      \end{cases}
     \end{equation}
     We now proceed by following the steps described above. \bigskip
     
     \noindent\textbf{Step I.} In this first step we prove the following result: 
     \emph{there exists \emph{(}at least\emph{)} one so\-lu\-tion $u\in W^{1,1}(I)$
      of problem \eqref{eq.BVPtrunc}; this means, precisely, that
      \begin{itemize}
       \item the map $t\mapsto \Phi(k(t)\,u'(t))$ is in $W^{1,1}(I)$ and
     $$\big(\Phi(k(t)\,u'(t))\big)' = F_u(t) \qquad\text{for a.e.\,$t\in I$};$$
       \item $u(a) = \BB_a[u]$ and $u(b) = \BB_b[u]$.
      \end{itemize}
      Furthermore, the following
      higher-regularity assertions hold:
      \begin{itemize}
       \item[\emph{(i)}] if
	$1/k\in L^\vartheta(I)$ for some $1<\vartheta\leq\infty$, 
	one also has that $x_0\in W^{1,\vartheta}(I)$;
       \item[\emph{(ii)}] if $k\in C(I,\R)$ and $k > 0$ on $I$, then $u\in C^1(I,\R)$.
      \end{itemize}}
 	\medskip
     
    \textsc{Claim 1.} \emph{There exists a non-negative function
      $\psi\in L^1(I)$ such that
      \begin{equation} \label{eq.essumptionF}
      |F_x(t)|\leq \psi(t) \qquad\text{for a.e.\,$t\in I$ and every $x\in W^{1,1}(I)$}.
      \end{equation}}
     First of all, by the choice of $M$
       and the very definition of $\TT_x$ we have
       $$-M\leq \alpha(t)\leq \TT_x(t)\leq\beta(t)\leq M \qquad\text{for all
       $x\in W^{1,1}(I)$ and any $t\in I$};$$
       as a consequence, owing to the choice of $\eta_M$ in \eqref{eq.choiceetaM},
       we get
       \begin{equation} \label{eq.boundforGGTTx}
        \|\GG_{\TT_x}\|_{L^\infty(I)}\leq \eta_M \qquad\text{for every $x\in W^{1,1}(I)$}.
       \end{equation}
       Moreover, owing to the very definition of
       $\DD$, we also have
       \begin{equation} \label{eq.boundforDDTTx}
        |\DD_{\TT_x'}(t)|\leq {\gamma}_L(t) \qquad\text{for any $x\in W^{1,1}(I)$
        and a.e.\,$t\in I$}.
       \end{equation}
       Gathering together \eqref{eq.boundforGGTTx} and
       \eqref{eq.boundforDDTTx}, we deduce from
       assumption (H6) that there exists
       a non-negative function $h = h_{\eta_M,{\gamma}_L}\in L^1(I)$ such that
       \begin{equation} \label{eq.estimFxpsitousedopo}
        |F_x(t)| \leq \big|f\big(t,\GG_{\TT_x}(t)\big)\,
       \rho\big(t, \DD_{\TT_x'}(t)\big)\big|
       +\frac{\pi}{2}\leq h_{\eta_M,{\gamma}_L}(t)
       +\frac{\pi}{2},
       \end{equation}
       and this estimate holds for every $x\in W^{1,1}(I)$ and a.e.\,$t\in I$.
       Since, obviously, 
       $$\psi := h_{\eta_M,{\gamma}_L}+\pi/2\in L^1(I),$$ 
       we conclude
       at once that $F_x\in L^1(I)$ for every $x\in W^{1,1}(I)$
       (hence, $F$ maps $W^{1,1}(I)$ into $L^1(I)$) and that
       $F$ satisfies estimate \eqref{eq.essumptionF}. \medskip
       
       \textsc{Claim 2.} 
       \emph{$F$ is continuous from $W^{1,1}(I)$ into $L^1(I)$}. \medskip
       
       \noindent Let $x_0\in W^{1,1}(I)$ be fixed, and let
       $\{x_n\}_n\subseteq W^{1,1}(I)$ be a sequence
       converging to $x_0$ as $n\to\infty$. Moreover,
       let $\{u_k := x_{n_k}\}_k$ be any sub-sequence
       of $\{x_n\}_n$. 
       
       To demonstrate the continuity of 
       $F$ it suffices to show that, by choosing a further
       sub-sequence if necessary, one has
       \begin{equation} \label{eq.toprovecontF}
        \lim_{k\to\infty}F_{u_k} = F_{x_0} \qquad\text{in $L^1(I)$}.
       \end{equation}
      First of all we observe that, since $u_k\to x_0$ in $W^{1,1}(I)$
      as $k\to\infty$, we have
      \begin{equation} \label{eq.uktox0uniformly}
       \lim_{k\to\infty}u_k(t) = x_0(t) \qquad\text{uniformly for $t\in I$};
      \end{equation}
      moreover, by Lemma \ref{lem.continuityTT} we also have
      \begin{equation} \label{eq.TTuktoTTx0}
       \lim\limits_{k\to\infty}\TT_{u_k} = \TT_{x_0} \qquad \text{in $W^{1,1}(I)$}.
       \end{equation} 
       In particular, since \eqref{eq.TTuktoTTx0} 
       implies that $\TT_{u_k}'\to \TT_{x_0}'$ in $L^1(I)$ as $k\to\infty$,
       by possibly choosing a sub-sequence we can assume that
       \begin{equation} \label{eq.DerTTuktoDerTTx0}
        \lim_{k\to\infty}\TT_{u_k}'(t) = \TT_{x_0}'(t) \qquad
        \text{for a.e.\,$t\in I$}. 
       \end{equation}
       Now, since $\GG$ is continuous from $W^{1,1}(I)$ to $L^\infty(I)$, from 
       \eqref{eq.TTuktoTTx0}
        we get
        \begin{equation} \label{eq.GGuktoGGx0}
        \lim_{k\to\infty}\GG_{\TT_{u_k}}(t) = \GG_{
        \TT_{x_0}}(t)\qquad\text{for every $t\in I$}.
      \end{equation}       
      Moreover, from \eqref{eq.DerTTuktoDerTTx0} we easily derive that
      \begin{equation} \label{eq.DDTTuktoDDTTx0}
       \lim_{k\to\infty}\DD_{\TT_{u_k}'}(t) = \DD_{\TT_{x_0}'}(t) \qquad\text{for a.e.\,$t\in I$}.
      \end{equation}
      Gathering together \eqref{eq.GGuktoGGx0},
      \eqref{eq.DDTTuktoDDTTx0} and \eqref{eq.uktox0uniformly}, 
      we then obtain
      (remind that, by as\-sum\-ptions, $f$ and $\rho$
      are Carath\'{e}odory functions on $I\times\R$)
      \begin{align*}
       \lim_{k\to\infty}F_{u_k}(t) & =
	 \lim_{k\to\infty}\big(-f\big(t,\GG_{\TT_{u_k}}(t)\big)\,\rho\big(t, \DD_{\TT_{u_k}'}(t)\big)+
      \arctan\big({u_k}(t)-\TT_{u_k}(t)\big) \\[0.1cm] 
       & 
       = -f\big(t,\GG_{\TT_{x_0}}(t)\big)\,\rho\big(t, \DD_{\TT_{x_0}'}(t)\big)+
      \arctan\big({u_0}(t)-\TT_{u_0}(t)\big) \\[0.1cm] 
       & = F_{x_0}(t) \qquad\quad\text{for a.e.\,$t\in I$}.
      \end{align*}
      From this, a standard dominated-convergence
      based on \eqref{eq.estimFxpsitousedopo}
      allows us to conclude that $F_{u_k}\to F_{x_0}$ in $L^1(I)$ as $k\to\infty$,
      which is exactly the desired \eqref{eq.toprovecontF}. \medskip
      
      \textsc{Claim 3.} \emph{$\BB_a$ and $\BB_b$ 
      are continuous and bounded} (\emph{from $W^{1,1}(I)$ to $\R$}).  \medskip
      
      \noindent As regards the continuity, since
      $\HH_a,\,\HH_b$ are continuous from $W^{1,1}(I)$ to $\R$
      (by assumption (H4))
      and since $\TT$ is continuous
      on $W^{1,1}(I)$ (by Lemma \ref{lem.continuityTT}), we im\-me\-dia\-te\-ly deduce that
      $\BB_a = \HH_a\circ\TT$ and $\BB_b = \HH_b\circ\TT$ are continuous.
      
      As regards the boundedness, since $\HH_a,\,\HH_b$
      are monotone increasing (see \eqref{eq.assumptionHHaHHbmonotone}), 
      for every fixed $x\in W^{1,1}(I)$ we have
      $$\HH_a[\alpha]\leq \HH_a[\TT_x]\leq \HH_a[\beta]\qquad\text{and}\qquad
      \HH_b[\alpha]\leq \HH_b[\TT_x]\leq \HH_b[\beta]
      $$
      (remind that, by definition, $\alpha\leq\TT_x\leq\beta$ for all $x\in W^{1,1}(I)$).
      From
      this, we im\-me\-dia\-te\-ly deduce that $\BB_a,\,\BB_b$ are globally bounded,
      and the claim is proved.
      \medskip
      
      \noindent Using the results established in the above claims, one can
      prove the existence of solutions for
      \eqref{eq.BVPtrunc} (and the higher-regularity assertions (i)-(ii))
      by arguing es\-sen\-tial\-ly as in 
      \cite[Lem.\,2.1 and Thm.\,2.2]{CaMaPa2018}.
     The key points are the following.
     \begin{itemize}
      \item Thanks to Claim 3, it can be proved
      that for every $x\in W^{1,1}(I)$ there exists
      a \emph{unique} real number $z = z_x\in\R$ such that
      \begin{align*}
       & \BB_b[x]-\BB_a[x] = \int_a^b\frac{1}{k(t)}\,\Phi^{-1}\big(
     z_x+\mathcal{F}_x(t)\big)\,\d t,
     \end{align*}
     where
     $\mathcal{F}_x(t) := \int_a^t F_x(s)\,\d s$.
     Moreover, the map $x\mapsto z_x$ is \emph{bounded}, i.e.,
     $$|z_x|\leq \c_0 \qquad\text{for every $x\in W^{1,1}(I)$},$$
    where $\c_0 > 0$ is a universal constant which is independent of $x$.
    
    \item The solutions of \eqref{eq.BVPtrunc}
    are precisely the fixed points (in $W^{1,1}(I)$) of the o\-pe\-ra\-tor
    $\mathcal{A}:W^{1,1}(I)\to W^{1,1}(I)$ defined as follows:
    $$
    \mathcal{A}_x(t)  := \BB_a[x]+\int_a^t
    \frac{1}{k(t)}\,\Phi^{-1}\big(
     z_x+\mathcal{F}_x(t)\big)\,\d t \qquad
     (t\in I).$$
     
     \item Using all the above claims, it can be proved
     that $\AA$ is continuous, bounded and compact on $W^{1,1}(I)$;
     thus, Schauder's Fixed-Point theorem ensures that
   $\AA$ possesses (at least) one fixed point
   $x_0\in W^{1,1}(I)$.
   
   \item Finally, the 
   higher-regularity assertions (i)-(ii)
   are straightforward con\-se\-quen\-ces of the 
	following simple 
   observations: 
   \begin{itemize}
   \item[$\diamond$]
   $\AA(W^{1,1}(I))\subseteq W^{1,\vartheta}(I)$ if 
   $1/k\in L^\vartheta(I)$ (for some $1<\vartheta \leq\infty$);
   \item[$\diamond$]
   $\AA(W^{1,1}(I))\subseteq C^1(I,\R)$ if 
    $k\in C(I,\R)$ and $k > 0$ on $I$.
     \end{itemize}
   \end{itemize}
   We proceed with the second step. \bigskip
   
   \noindent\textbf{Step II.} In this second step we establish the following result:
   \emph{if $u\in W^{1,1}(I)$ is \emph{any} solution of
        \eqref{eq.BVPtrunc},
        then $u$ is also a solution of \eqref{eq.BVPmainSec2}.} \bigskip
        
    \textsc{Claim 1.} \emph{$\alpha(t)\leq u(t)\leq\beta(t)$ for every $t\in I$, so that}
          $$\text{\emph{$\TT_u\equiv u$ and $\GG_{\TT_u}\equiv \GG_u$ on $I$}}.$$
	\noindent We argue by contradiction and, to fix ideas,
        we assume that the (con\-ti\-nuous) function
        $v := u-\alpha$ attains a strictly negative minimum on $I$.
        
        Since $u$ solves \eqref{eq.BVPtrunc},
        $\alpha$ is a lower solution of problem \eqref{eq.BVPmainSec2}
        and the operators $\HH_a,\,\HH_b$ are monotone increasing (see
         assumption (H4)), we get
         $$u(a) = \HH_a[\TT_u] \geq \HH_a[\alpha] \geq \alpha(a)\quad \text{and}\quad
         u(b) = \HH_b[\TT_u] \geq \HH_b[\alpha] \geq \alpha(b)$$
         (remind that, by definition,
         $\TT_u\geq \alpha$ on $I$). 
         As a consequence, it is possible to find three 
         points $t_1,t_2,\theta\in \mathrm{int}(I)$,
         with $t_1<\theta<t_2$, such that \medskip
         
         (1)\,\,$u(t_i)-\alpha(t_i) = 0$ for $i = 1,2$; \vspace*{0.1cm}
         
         (2)\,\,$u(t)-\alpha(t) < 0$ for all $t\in (t_1,t_2)$; \vspace*{0.1cm}
         
         (3)\,\,$u(\theta)-\alpha(\theta) = \min\limits_{t\in I}(u(t)-\alpha(t)) < 0$. \medskip

		\noindent In particular, from (2) we infer that $\TT_u \equiv \alpha$ on $(t_1,t_2)$,
		and thus
		$$\DD_{\TT_u'}(t) = \DD_{\alpha'}(t) = \alpha'(t) \qquad\text{for a.e.\,$t\in I$}.$$
		By using once again the fact that $u$ solves \eqref{eq.BVPtrunc},
		and since $\alpha$ is a lower solution
		of problem \eqref{eq.BVPmainSec2},
		 from assumption (H3) we then obtain (for a.e.\,$t\in (t_1,t_2)$)
		\begin{equation} \label{eq.keycomputation}
		 \begin{split}
		   \big(\Phi(k(t)\,u'(t))\big)' & = -
		   f(t,\GG_{\TT_u}(t))\,\rho(t, \alpha'(t)) + \arctan(u(t)-\alpha(t)) \\[0.1cm]
		   & < -
		   f(t,\GG_{\TT_u}(t))\,\rho(t, \alpha'(t)) \\[0.1cm]
		   & = -\big(f(t,\GG_{\TT_u}(t))+\kappa\,\TT_u(t)\big)\rho(t, \alpha'(t)) 
		   + \kappa\,\TT_u(t)\,\rho(t, \alpha'(t)) \\[0.1cm]
		   & \big(\text{by \eqref{eq.assumptionGGmonotone}, since 
		   $\TT_u\geq \alpha$ on $I$ and $\rho\geq 0$ on $\R$}\big) \\[0.1cm]
		   & \leq 
		   -\big(f(t,\GG_{\alpha}(t))+\kappa\,\alpha(t)\big)\rho(t, \alpha'(t)) 
		   + \kappa\,\TT_u(t)\,\rho(t, \alpha'(t)) \\[0.1cm]
		   & \big(\text{since $\TT_u \equiv \alpha$ on $(t_1,t_2)$}\big) \\[0.1cm]
		   & = -f(t,\GG_{\alpha}(t))\,\rho(t, \alpha'(t)) 
		   \leq 
		   \big(\Phi(k(t)\,\alpha'(t))\big)'.
		 \end{split}
		\end{equation}
		We now consider the sets $A_1,\,A_2\subseteq I$ defined as follows:
		\begin{align*}
		& A_1 := \big\{t\in (t_1,\theta):\,\text{$\exists\,\,u'(t),\,\alpha'(t)$ and
		$u'(t) < \alpha'(t)$}\big\}\qquad\text{and} \\[0.1cm]
		& \qquad A_2 := \big\{t\in (\theta,t_2):\,\text{$\exists\,\,u'(t),\,\alpha'(t)$ and
		$u'(t) > \alpha'(t)$}\big\}.
		\end{align*}
		Since $u,\,\alpha\in W^{1,1}(I)$ and since $u<\alpha$ on $(t_1,t_2)$, it is
		very easy to check that both $A_1$ and $A_2$ have \emph{positive}
		Lebesgue measure; as a consequence, there exist
		$\tau_1\in A_1$ and $\tau_2\in A_2$ such that
		(see also Remarks \ref{rem.assumptionkconcrete}
		and \ref{rem.functionKx}) \medskip

		 (a)\,\,$k(\tau_i) > 0$ for $i = 1,2$; \vspace*{0.1cm}
		 
		 (b)\,\,$\KK_u(\tau_i) = k(\tau_i)\,u'(\tau_i)$ for $i = 1,2$; \vspace*{0.1cm}

		 (c)\,\,$\KK_\alpha(\tau_i) = k(\tau_i)\,\alpha'(\tau_i)$ for $i = 1,2$. \medskip

		\noindent By integrating both sides of
	    \eqref{eq.keycomputation} on $[\tau_1,\theta]$, and using (b)-(c),
		we then get
		$$\Phi\big(\KK_u(\theta)\big)- 
		\Phi\big(k(\tau_1)\,u'(\tau_1)\big)
		< \Phi\big(\KK_\alpha(\theta)\big) - 
		\Phi\big(k(\tau_1)\,\alpha'(\tau_1)\big).$$
        Since 
		$\Phi$ is strictly increasing, by (a) and the choice of $\tau_1$ we obtain
		\begin{equation} \label{eq.inequalityabsurd}
		 \Phi\big(\KK_u(\theta)\big)-\Phi\big(\KK_\alpha(\theta)\big) < 0. 
		\end{equation}
		On the other hand, by 
	    integrating both sides of inequality
	    \eqref{eq.keycomputation} on $[\theta,\tau_2]$ (and using once again
	    (b)-(c)),
		we derive that
		$$
		\Phi\big(k(\tau_2)\,u'(\tau_2)\big)-
		\Phi\big(\KK_u(\theta)\big)
		< 
		\Phi\big(k(\tau_2)\,\alpha'(\tau_2)\big)
		- \Phi\big(\KK_\alpha(\theta)\big);$$
		Since $\Phi$ is strictly increasing, by (a) and the choice of $\tau_2$ we obtain
		$$\Phi\big(\KK_u(\theta)\big)-\Phi\big(\KK_\alpha(\theta)\big) > 0.$$ 
		This is clearly in contradiction with 
		\eqref{eq.inequalityabsurd}, and thus $u(t)- \alpha(t)\geq 0$ 
		for every $t\in I$.
		By arguing exactly in the same way one can also
		prove that $u(t)- \beta(t)\leq 0$ for every $t\in I$,
		and the claim is completely demonstrated. \medskip
    
        \textsc{Claim 2.} 
        \emph{$|u(t)|\leq M$ and $|\GG_u(t)|\leq \eta_M$ for every $t\in I$}.
        \medskip
        
        \noindent By statement (i) and the choice of $M\geq \|\alpha\|_{L^\infty(I)},\,\|
		\beta\|_{L^\infty(I)}$, we get
		$$-M\leq \alpha(t)\leq u(t)\leq\beta(t)\leq M\qquad\text{for every $t\in I$},$$
		and this proves that $|u(t)|\leq M$ for all $t\in I$. From this,
		by taking into account the choice of $\eta_M$ in \eqref{eq.choiceetaM},
		we derive that $|\GG_u(t)|\leq\eta_M$, as desired.
		\medskip
		
		\textsc{Claim 3.} \emph{If $N > 0$ is as in
		\eqref{eq.choiceN}, then}
		\begin{equation} \label{eq.minKKuleqN}
		 \min_{t\in I}|\KK_u(t)|\leq N.
		\end{equation}
		By contradiction, let us assume that
		\eqref{eq.minKKuleqN} does not hold; moreover, to fix ideas
		(and taking into account the continuity of
		$\KK_u$),
		let us suppose that 
		\begin{equation} \label{eq.minKKabsurd}
		 \KK_u(t) > N \qquad\text{for every $t\in I$}.
		 \end{equation} 
		By integrating on $I$ both sides of the above inequality, we obtain
		\begin{align*}
		 N(b-a) < \int_a^b\KK_u(t)\,\d t = \int_a^bk(t)\,u'(t)\,\d t = (\bigstar);
		\end{align*}
		from this, since \eqref{eq.minKKabsurd}
		implies that $u'(t) > N/k(t)$ for a.e.\,$t\in I$, we then get
		\begin{align*}
		 (\bigstar) & \leq \|k\|_{L^\infty(I)}\,\int_a^bu'(t)\,\d t
		 = \|k\|_{L^\infty(I)}\,(u(b)-u(a)) \\
		 & \big(\text{by statement (ii) and
		 the choice of $N$, see \eqref{eq.choiceN}}\big) \\[0.1cm]
		 & \leq (2M)\cdot\|k\|_{L^\infty(I)} < N(b-a).
		\end{align*}
		This is clearly a contradiction, and thus
		$\min_I\KK_u\leq N$. By arguing exactly in the same way one can also
		show that $\sup_I\KK_u\geq -N$, and this proves 
	    \eqref{eq.minKKuleqN}. \medskip
	    
	    \textsc{Claim 4}. 
	    \emph{$|\KK_u(t)|\leq L_M$ for every $t\in I$}. \medskip
	    
	    \noindent Ar\-gu\-ing again by con\-tra\-diction,
	    we assume that there exists some point $\tau$ in $I$
	    such that $|\KK_u(\tau)| > L_M$; moreover, to fix ideas, 
	    we suppose that
	    \begin{equation*}
	     \KK_u(\tau) > L_M > 0.
	    \end{equation*}
	    Since $L_M > N$ (see \eqref{eq.choiceL}), by
	    \eqref{eq.minKKuleqN} (and the continuity of
	    $\KK_u$) we deduce the existence of
	    two points $t_1,\,t_2\in I$, with (to fix ideas)
	    $t_1<t_2$, such that \medskip
	    
	    (a)\,\,$\KK_u(t_1) = N$ and $\KK_u(t_2) = L_M$; \vspace*{0.1cm}
	    
	    (b)\,\,$N < \KK_u(t) < L_M$ for all $t\in (t_1,t_2)\subseteq I$. \medskip
	    
	    \noindent In particular, from (b), \eqref{eq.defigammaL}
	    and the choice
	    of $N$ in \eqref{eq.choiceN} we derive that
	    \begin{equation} \label{eq.stimuprimeHLM}
	     k(t)u'(t) \geq H_M\quad\text{and}\quad 0 < u'(t) < \frac{L_M}{k(t)}
	     \leq{\gamma}_L(t)
	    \end{equation}
	    for almost every $t\in (t_1,t_2)$. Now,
	    on account of \eqref{eq.stimuprimeHLM} and of the very
	    definition of $\DD$, we deduce that
	    $\DD_{u'}\equiv u'$ on $(t_1,t_2)$; as a consequence,
	    since $u$ is a solution of
	    \eqref{eq.BVPtrunc} and $\TT_u\equiv u$ on $I$ (by Claim 1.),
	    we have (a.e.\,on $(t_1,t_2)$)
	    \begin{align*}
	     & \big|\big(\Phi(\KK_u(t))\big)'\big|
	     = \big|\big(\Phi(k(t)\,u'(t))\big)'\big| = \big|f(t,\GG_u(t))\,\rho(t, u'(t))\big|
	     \\[0.1cm]
	     & \qquad \big(\text{by \eqref{eq.Nagumocondition}, 
	     since $k(t)u'(t) > H_M$ and $\|\GG_u\|\leq\eta_M$, see (ii)}\big) \\[0.1cm]
	     & \qquad \leq 
	     \psi_M\big(|\Phi(k(t)\,u'(t))|\big)
	 \cdot\big(l_M(t)+\mu_M(t)\,|u'(t)|^{\frac{q-1}{q}}\big) \\[0.1cm]
	 & \qquad = \psi_M\big(|\Phi(\KK_u(t))|\big)
	 \cdot\big(l_M(t)+\mu_M(t)\,|u'(t)|^{\frac{q-1}{q}}\big). 
	    \end{align*}
       In particular, since $u' > 0$ a.e.\,on $(t_1,t_2)$ (see
         \eqref{eq.stimuprimeHLM}) and since
         $$\Phi(\KK_u(t)) > \Phi(N) > 0\qquad\text{for any $t\in(t_1,t_2)$}$$ 
         (by (b), the monotonicity of
         $\Phi$ and the choice of $N$ in \eqref{eq.choiceN}), we obtain 
         \begin{equation} \label{eq.toIntegrate}
          \big|\big(\Phi(\KK_u(t))\big)'\big|
          \leq 
          \psi_M\big(\Phi(\KK_u(t))\big)
	 \cdot\big(l_M(t)+\mu_M(t)\,(u'(t))^{\frac{q-1}{q}}\big)
         \end{equation}
         for almost every $t\in (t_1,t_2)$.
         Using this last inequality, we then get
         (remind that $\Phi\circ\KK_u$ is absolutely continuous,
         see Remark \ref{rem.functionKx})
         \begin{align*}
           \int_{\Phi(N)}^{\Phi(L_M)}
          \frac{1}{\psi_M}\,\d s & = 
          \int_{\Phi(\KK_u(t_1))}^{\Phi(\KK_u(t_2))}\frac{1}{\psi_M}\,\d s
          = \int_{t_1}^{t_2}\frac{\big(\Phi(\KK_u(t))\big)'}
          {\psi_M\big(\Phi(\KK_u(t))\big)}\,\d t \\[0.1cm]
          & 
          \leq \int_{t_1}^{t_2}\big(l_M(t)+\mu_M(t)\,(u'(t))^{\frac{q-1}{q}}\big)\,\d t
          \\[0.1cm]
          & \leq \|l_M\|_{L^1(I)}+\int_{t_1}^{t_2}\mu_M(t)\,(u'(t))^{\frac{q-1}{q}}\,\d t
          \\[0.1cm]
          & \big(\text{by H\"older's inequality, since $\mu_M\in L^q(I)$}\big) \\[0.15cm]
          & \leq \|l_M\|_{L^1(I)}+\|\mu_M\|_{L^q(I)}\,(u(t_2)-u(t_1))^{\frac{q-1}{1}}\\[0.15cm]
          & \big(\text{since $|u(t)|\leq M$ for all $t\in I$, see Claim 2.}\big) \\[0.15cm]
          & \leq \|l_M\|_{L^1(I)}+\|\mu_M\|_{L^q(I)}\,
          (2M)^{\frac{q-1}{1}}.
         \end{align*}
         This is in contradiction with the choice of $L_M$ in \eqref{eq.choiceL},
         and thus $\KK_u(t)\leq L_M$ for every $t\in I$. By arguing exactly
         in the same way one can also show that $\KK_u(t)\geq -L_M$
         for all $t\in I$, and the claim is completely proved. \medskip
    
         \textsc{Claim 5.} 
         \emph{$|u'(t)|\leq L_M/k(t)$ for a.e.\,$t\in I$,
          so that} 
          $$\text{\emph{$\DD_{u'}\equiv u'$ a.e.\,on $I$}.}$$
          By Claim 4 and the very definition
         of $\KK_u$ we immediately get
         $$|u'(t)| = \frac{|\KK_u(t)|}{k(t)}\leq\frac{L_M}{k(t)}\qquad\text{
         for a.e.\,$t\in I$};$$
         as a consequence, since $L_M/k(t) \leq \gamma_L(t)$
         a.e.\,on $I$
         (see \eqref{eq.defigammaL}), from
         the very de\-fi\-ni\-tion of $\DD$
         in \eqref{eq.defiOpTT} we conclude that
         $\DD_{u'}\equiv u'$ on $I$. \bigskip
         
          Using the results
         established in the above claims, we can complete the proof
         of this step. Indeed, by Claim 1.\,we have 
         $\TT_u\equiv u$ and $\GG_{\TT_u}\equiv \GG_u$ on $I$;
         moreover, by Claim 5.\,we know
         that $\DD_{u'}\equiv u'$ a.e.\,on $I$. 
         Gathering to\-ge\-ther all these facts (and since
         $u$ is a solution of \eqref{eq.BVPtrunc}), for almost every
         $t\in I$ we get
         \begin{align*}
          \big(\Phi(\KK_u(t))\big)' & =
          -f(t,\GG_{\TT_u}(t))\,\rho(t, \DD_{\TT_u'}(t))
          + \arctan\big({u}(t)-\TT_{u}(t)\big) \\[0.1cm]
          &  = -f(t,\GG_u(t))\,\rho(t, u'(t)),
         \end{align*}
         and thus $u$ solves the ODE \eqref{eq.ODEconcrete}. Furthermore,
         by \eqref{eq.defiBBtrunc} we have
         \begin{align*}
         & u(a) = \BB_a[u] = \HH_a[\TT_u] = \HH_a[u] \qquad\text{and} \\[0.1cm]
         & \qquad u(b) = \BB_b[u] = \HH_b[\TT_u] = \HH_b[u],
         \end{align*}
         and this proves that $u$ is a solution of the BVP \eqref{eq.BVPmainSec2}.
         \bigskip
         
         \noindent Thanks to the results in Steps I and II, we are finally in a po\-si\-tion
         to conclude the proof of Theorem \ref{thm.mainPAPER}.
         Indeed, by Step I we know that there
        exists (at least) one solution $x_0\in W^{1,1}(I)$ of the 
        truncated BVP \eqref{eq.BVPtrunc}; on the other hand, we
        derive from Step II that
        $x_0$ is actually a solution of \eqref{eq.BVPmainSec2}. \vspace*{0.07cm}
        
        To proceed further we observe that, owing to
        Claim 1.\,in Step II,
        we im\-me\-dia\-te\-ly derive that
        $x_0$ satisfies \eqref{eq.x0traalphabeta}; moreover, the result
        in Step I ensures that
        \begin{itemize}
         \item if $1/k\in L^\vartheta(I)$ for some $1<\vartheta\leq\infty$, then
         $x_0\in W^{1,\vartheta}(I)$;
         \item if $k\in C(I,\R)$ and $k > 0$ on $I$, then $x_0\in C^1(I,\R)$.
        \end{itemize}
        Finally, by combining Claims 2.\,and 5.\,in
        Step II, we conclude
        that $x_0$ satisfies the `a-priori'
        estimate \eqref{eq.boundKKx0LM}, and the proof
        is complete.\hfill$\square$
    \begin{remark} \label{rem.assumptionH2Lp}
        By carefully scrutinizing the proof Theorem \ref{thm.mainPAPER}, one can recognize that
        estimate \eqref{eq.growthfrhoass} in assumption (H6)
        has been used only for demonstrating the result
        in Step I, and with the specific choice
        $$R = \eta_M \qquad\text{and}\qquad\gamma(t) = {\gamma}_L(t) = 
        \frac{L_M}{k(t)}+|\alpha'(t)|+|\beta'(t)|.$$
        As a consequence, if we know that
        \begin{equation} \label{eq.gammaLLtheta}
         {\gamma}_L\in L^\vartheta(I) \qquad(\text{for some $\vartheta > 1$}), 
        \end{equation}
        assumption (H6) can be replaced by the following
        weaker one:
        \begin{itemize}
         \item[{(H6)'}] for every $R > 0$ and every non-negative
       function $\gamma\in L^\vartheta(I)$ there exists a non-negative
       function $h = h_{R,\gamma}\in L^1(I)$ such that
       \begin{equation} \label{eq.growthfrhoassWEAK}
        \begin{array}{c}
        |f(t,z)\,\rho(t, y(t))|\leq h_{R,\gamma}(t) \\[0.2cm]
        \text{for a.e.\,$t\in I$, every $z\in\R$ with $|z|\leq R$} \\[0.05cm]
        \text{and every
        $y\in L^\vartheta(I)$ such that $|y(s)|\leq \gamma(s)$ for a.e.\,$s\in I$}
        \end{array}
       \end{equation}
        \end{itemize}
        Notice that \eqref{eq.gammaLLtheta}
       is certainly satisfied if
       $1/k\in L^\vartheta(I)$ and if the lo\-wer/up\-per so\-lu\-tions
       $\alpha,\beta$ in assumption (H5) can be chosen in $W^{1,\vartheta}(I)$.
       \end{remark}
      
     \section{Some examples} \label{sec.examples}
      In this last section of the paper
      we present some `model BVPs' illustrating
      the applicability of our existence result in Theorem \ref{thm.mainPAPER}.
      \begin{example} \label{exm.linear}
       Let us consider the following BPV on $I = [0,1]$
       \begin{equation} \label{eq.PBexm1}
        \begin{cases}
         \big(\sinh
         \big(\sqrt{t(1-t)}\cdot x'(t)\big)\big)' + 
         a\Big(\int_{0}^{t}x^3(s)\,\d s\Big)\,|x'(t)|^\varrho = 0
         & \text{a.e.\,on $I$}, \\[0.25cm]
         \,\,x(0) = \max\big\{x(1),1\big\},\,\,
         x(1) = \varepsilon\int_0^1x(s)\,\d s,
        \end{cases}
       \end{equation}
       where $a:\R\to\R$ is a general continuous non-decreasing function
       and $\varepsilon,\varrho\in(0,1)$.
       Problem \eqref{eq.PBexm1} takes the form
       \eqref{eq.BVPmainSec2}, with \medskip
       
 		$(\ast)$\,\,$k:I\to\R, \quad k(t) := \sqrt{t(1-t)}$; \vspace*{0.1cm}
 		
        $(\ast)$\,\,$\Phi:\R\to\R, \quad \Phi(z) := \sinh(z)$; \vspace*{0.1cm}
        
        $(\ast)$\,\,$f:I\times\R\to\R,\quad f(t,z) := a(z)$; \vspace*{0.1cm}
        
        $(\ast)$\,\,$\rho:I\times\R\to\R, \quad \rho(t, y) := |y|^\varrho$; \vspace*{0.1cm}
        
        $(\ast)$\,\,$\GG_x(t) := \int_0^tx^3(s)\,\d s$ (for $x\in W^{1,1}(I)$); \vspace*{0.1cm}
        
        $(\ast)$\,\,$\HH_0[x] := \max\{x(1),1\}$
        and $\HH_1[x] := \varepsilon\int_0^1x(s)\,\d s$ (for $x\in W^{1,1}(I)$). \medskip
        
        \noindent We aim to show that \emph{all} the assumptions
        of Theorem \ref{thm.mainPAPER} are satisfied in this case,
        so that problem \eqref{eq.PBexm1}
        possesses (at least) one solution $x_0\in W^{1,1}(I)$. 
		We explicitly point out that, in view of the boundary conditions,
		$x_0$ \emph{cannot be constant}.
        \medskip
        
         To begin with, we observe assumption
        (H1) is tri\-vial\-ly satisfied, since
        \begin{equation} \label{eq.sumkexm1}
         \text{$1/k\in L^\vartheta(I)$ for all $\vartheta\in[1,2)$}.
        \end{equation}
        As regards assumptions (H2)-(H3), we first notice that $\GG$ is a continuous
        operator mapping $W^{1,1}(I)$ into $C^1(I,\R)$; as a consequence,
        owing to Remark \ref{rem.continuityGG}, we know that
        $\GG$ is continuous from $W^{1,1}(I)$ into 
        $L^\infty(I)$ (with the usual norms).

        Furthermore, if $r > 0$ is any fixed positive number, we have 
        $$\|\GG_x\|_{L^\infty(I)}\leq \int_0^1|x(t)|^3\,\d t
        \leq r^3\quad\text{for all $x\in W^{1,1}$ with $\|x\|_{L^\infty}\leq r$},$$
        and thus \eqref{eq.assumptionGGbounded} is satisfied with $\eta_r := r^3$.
        Finally, since $a$ is non-decreasing and $\GG$ is increasing
        (with respect to the point-wise order), we readily see that
        $$W^{1,1}(I)\ni x \mapsto f(t,\GG_x(t)) = a\bigg(\int_0^tx^3(s)\,\d s\bigg)$$
        is monotone increasing, so that \eqref{eq.assumptionGGmonotone}
        holds with $\kappa = 0$. \vspace*{0.15cm}
        
        As regards assumption (H4), it is very easy to check that
        $\HH_0,\,\HH_1$ are continuous from
        $W^{1,1}(I)$ to $\R$ (remind that $W^{1,1}(I)$ is continuously
        embedded into $C(I,\R)$); moreover, if $x,y\in W^{1,1}(I)$
        are such that $x\leq y$ point-wise on $I$, then
        \begin{align*}
         & \HH_0[x] = \max\big\{x(1),1\big\}\leq
         \max\big\{y(1),1\big\} = \HH_0[y]\qquad\text{and} \\[0.1cm]
         & \qquad \HH_1[x] = \varepsilon\int_0^1x(s)\,\d s\leq 
         \varepsilon\int_0^1y(s)\,\d s = \HH_1[y],
        \end{align*}
        so that $\HH_0,\,\HH_1$ are also monotone increasing
        (w.r.t.\,to the point-wise order). \medskip
        
        \noindent We now turn to prove the validity of assumptions
        (H5)-to-(H7). \medskip
        
        \textsc{Assumption (H5).} We claim that the constant functions
        $$\alpha(t) := -1 \qquad\text{and}\qquad\beta(t) := 1$$
        are, respectively, a lower and an upper solution
        of problem \eqref{eq.PBexm1}. 
        
        In fact, 
        since $\rho(t, 0) = 0$ for all $t\in I$, we
        know from Remark \ref{rem.existencealphabeta}
        that $\alpha$ and $\beta$ are both lower and upper
        solutions of the \emph{differential equation}
	    $$\big(\sqrt{t(1-t)}\cdot x'(t)\big)\big)' 
        + a\Big(\int_{0}^{t}x^3(s)\,\d s\Big)\,|x'(t)|^\varrho = 0;$$
		moreover, owing to the very definitions of $\HH_0$ and $\HH_1$ we have \medskip
		
		(a)\,\,$\alpha(0) = -1 \leq 1 = 
		\HH_0[\alpha]$\,\,and\,\,$\alpha(1) = -1 \leq -\varepsilon = \HH_1[\alpha]$; \vspace*{0.15cm}
		
		(b)\,\,$\beta(0) = 1 = \HH_0[\beta]$\,\,and\,\,$\beta(1) = 1 \geq \varepsilon = \HH_1[\beta]$.
		\medskip
		
        \noindent On account of
        Definition \ref{def.lowerupper}, from (a)-(b) 
        we immediately derive that $\alpha$ is a lower 
        solution and $\beta$ is an upper solution of
        \emph{problem} \eqref{eq.PBexm1}. \medskip
        
        \textsc{Assumption (H6).} Let $R > 0$ be fixed
        and let $\gamma$ be a non-negative function belonging
        to $L^1(I)$. Since, by assumption, $a\in C(\R,\R)$, we have
        $$\begin{array}{c}
         |f(t,z)\,\rho(t, y(t))| = |a(z)|\cdot|y(t)|^\varrho\leq
        \big(\max_{|z|\leq R}|a(z)|\big)\cdot\gamma(t)^\varrho =: h_{R,\gamma}(t) \\[0.25cm]
        \text{for a.e.\,$t\in I$, every $z\in\R$ with $|z|\leq R$} \\[0.05cm]
        \text{and every
        $y\in L^1(I)$ such that $|y(s)|\leq \gamma(s)$ for a.e.\,$s\in I$}.
        \end{array}
        $$
        As a consequence, since $h_{R,\gamma}\in L^1(I)$
        (remind that, by assumption,
        $0<\varrho<1$), we immediately conclude that
        estimate \eqref{eq.growthfrhoass} is satisfied. \medskip
        
        \textsc{Assumption (H7).} Let $R > 0$ be arbitrarily fixed. 
        Since, by assumption, $a\in C(\R,\R)$,
        we have the following estimate
        \begin{align*}
         |f(t,z)\,\rho(t, y)| & = |a(z)|\cdot|y|^\varrho
         \leq \big(\max_{|z|\leq R}|a(z)|\big)\cdot
         |y|^\varrho,
        \end{align*}
        holding true for a.e.\,$t\in I$, every $z\in[-R,R]$
        and every $y\in\R$. As a consequence, we conclude that
        estimate \eqref{eq.Nagumocondition} is satisfied with the choice
        $$H_R = 1, \quad \psi_R\equiv 1, \quad
        l_R(t) \equiv 0,\quad
        \mu_R := \big(\max_{|z|\leq R}|a(z)|\big), \quad q = \frac{1}{1-\varrho}.$$
        Since all the assumptions of Theorem \ref{thm.mainPAPER} are fulfilled,
        we can conclude that there exists (at least) one solution
        $x_0\in W^{1,1}(I)$ of problem \eqref{eq.PBexm1},
        further satisfying
        $$-1\leq x_0(t)\leq 1\qquad\text{for every $t\in I$}.$$
        Moreover, from \eqref{eq.sumkexm1} we deduce that $x_0\in W^{1,\vartheta}(I)$
        for all $\vartheta\in [1,2)$.  
      \end{example}
      \begin{example} \label{exm.plaplacian}
	  Let $\vartheta_0\in(1,\infty)$ be fixed,
	  and let $\tau\in (0,2\pi)$. Moreover, let $p,\delta\in\R$ be 
	  two positive real numbers satisfying the following relation
      \begin{equation} \label{eq.choiceralphaexm2}
       1 < p < \vartheta_0+1 \qquad\text{and}\qquad
       0 < \delta < p-\frac{p-1}{\vartheta_0}.
      \end{equation}
       Finally, let $\Phi_p(z) := |z|^{p-2}z$ be usual
       $p$-Laplace operator on $\R$. We then
      consider 
       the following boundary-value problem on $I = [0,2\pi]$
       \begin{equation} \label{eq.PBexm2}
        \begin{cases}
         \big(\Phi_p
         \big(|\sin(t)|^{1/\vartheta_0}\,x'(t)\big)\big)' + 
         x_{\tau}(t)\,|x'(t)|^\delta = 0
         & \text{a.e.\,on $I$}, \\[0.25cm]
         \,\,x(0) = \sqrt[3]{x(\pi)},\,\,
         x(2\pi) = \frac{1}{4\pi}\int_0^{2\pi}({x(s)}+{2})\,\d s
        \end{cases}
       \end{equation}
       where $x_\tau$ is the delay-type function defined as
       $$x_\tau(t) := \begin{cases}
       x(t-\tau), & \text{if $\tau\leq t\leq 2\pi$}, \\
       x(0), & \text{if $0\leq t < \tau$}.
       \end{cases}$$
       Problem \eqref{eq.PBexm2} takes the form
       \eqref{eq.BVPmainSec2}, with \medskip
       
 		$(\ast)$\,\,$k:I\to\R, \quad k(t) := |\sin(t)|^{1/\vartheta_0}$; \vspace*{0.1cm}
 		
        $(\ast)$\,\,$\Phi:\R\to\R, \quad \Phi(z) = \Phi_p(z) = |z|^{p-2}z$; \vspace*{0.1cm}
        
        $(\ast)$\,\,$f:I\times\R\to\R,\quad f(t,z) := z$; \vspace*{0.1cm}
        
        $(\ast)$\,\,$\rho:I\times\R\to\R, \quad \rho(t, y) := |y|^\delta$; \vspace*{0.1cm}
        
        $(\ast)$\,\,$\GG_x(t) := x_{\tau}$ (for $x\in W^{1,1}(I)$); \vspace*{0.1cm}
        
        $(\ast)$\,\,$\HH_0[x] := \sqrt[3]{x(\pi)}$
        and $\HH_{2\pi}[x] := 
        \frac{1}{4\pi}\int_0^{2\pi}(x(s)+2)\,\d s$ (for $x\in W^{1,1}(I)$). \medskip
        
        \noindent We aim to show that \emph{all} the assumptions
        of Theorem \ref{thm.mainPAPER} are satisfied in this case,
        so that problem \eqref{eq.PBexm2}
        possesses (at least) one solution $x_0\in W^{1,1}(I)$. 
        We explicitly point out that, in view of the boundary conditions,
		$x_0$ \emph{cannot be constant}.
		\medskip
        
        To begin with, we observe that assumption (H1)
        is tri\-vial\-ly satisfied, since
        \begin{equation} \label{eq.sumkexm2}
         \text{$1/k\in L^\vartheta(I)$ for all $\vartheta\in[1,\vartheta_0)$}.
        \end{equation}
        As regards assumptions (H2)-(H3), we first notice that $\GG$ is a well-defined
        linear operator mapping $W^{1,1}(I)$ into 
        $L^\infty(I)$; as a consequence, since
        $$\|\GG_x\|_{L^\infty(I)} \leq \|x\|_{L^\infty(I)}
        \leq C\,\|x\|_{W^{1,1}(I)}\qquad\text{for every
        $x\in W^{1,1}(I)$},$$
        we immediately conclude that $\GG$ is continuous
        from $W^{1,1}(I)$ into $L^\infty(\R)$.
        Fur\-ther\-mo\-re, if $r > 0$ is any fixed positive number, we also have
        $$\|\GG_x\|_{L^\infty(I)}\leq \|x\|_{L^\infty(I)}
        \leq r\quad\text{for all $x\in W^{1,1}$ with $\|x\|_{L^\infty}\leq r$},$$
        and thus \eqref{eq.assumptionGGbounded} is satisfied with $\eta_r := r$.
        Finally, since $\GG$ is monotone increasing
        with respect to the point-wise order (as it is very easy to check)
        and since $f(t,z) = z$, one straightforwardly
        derives that \eqref{eq.assumptionGGmonotone}
        holds with $\kappa = 0$. \vspace*{0.15cm}
        
        As regards assumption (H4), it is easy to check that
        $\HH_0,\,\HH_{2\pi}$ are con\-ti\-nuo\-us from
        $W^{1,1}(I)$ to $\R$ (remind that $W^{1,1}(I)$ is continuously
        embedded into $C(I,\R)$); moreover, if $x,y\in W^{1,1}(I)$
        are such that $x\leq y$ point-wise on $I$, then
        \begin{align*}
         & \HH_0[x] = \sqrt[3]{x(\pi)}\leq \sqrt[3]{y(\pi)} = \HH_0[y]\qquad\text{and} \\[0.1cm]
         & \qquad \HH_{2\pi}[x] = \frac{1}{4\pi}\int_0^{2\pi}(x(s)+2)\,\d s\leq 
         \frac{1}{4\pi}\int_0^{2\pi}(y(s)+2)\,\d s = \HH_{2\pi}[y],
        \end{align*}
        so that $\HH_0,\,\HH_{2\pi}$ are also monotone increasing
        (w.r.t.\,to the point-wise order). \medskip
        
        \noindent We now turn to prove the validity of assumptions
        (H5)-to-(H7). \medskip
        
        \textsc{Assumption (H5).} We claim that the constant functions
        $$\alpha(t) := 1 \qquad\text{and}\qquad\beta(t) := 2$$
        are, respectively, a lower and an upper solution
        of problem \eqref{eq.PBexm1}. 
        
        In fact, 
        since $\rho(t,0) = 0$ for all $t\in I$, we
        know from Remark \ref{rem.existencealphabeta}
        that $\alpha$ and $\beta$ are both lower and upper
        solutions of the \emph{differential equation}
	    $$\big(\Phi_p
         \big(|\sin(t)|^{1/\vartheta_0}\,x'(t)\big)\big)' + 
         x_{\tau}(t)\,|x'(t)|^\delta = 0;$$
		moreover, owing to the very definitions of $\HH_0$ and $\HH_{2\pi}$ we have \medskip
		
		(a)\,\,$\alpha(0) = 1 =
		\HH_0[\alpha]$\,\,and\,\,$\alpha(2\pi) = 1 < 3/2 = \HH_{2\pi}[\alpha]$; \vspace*{0.15cm}
		
		(b)\,\,$\beta(0) = 2 > \HH_0[\beta]$\,\,and\,\,$\beta(2) = 2 = \HH_{2\pi}[\beta]$.
		\medskip
		
        \noindent On account of
        Definition \ref{def.lowerupper}, from (a)-(b) 
        we immediately derive that $\alpha$ is a lower solution
        and $\beta$ is an upper solution
        of
        \emph{problem} \eqref{eq.PBexm2}. \medskip
        
        \textsc{Assumption (H6).} We first observe that, by \eqref{eq.choiceralphaexm2}, we have
        $$0<\delta<\vartheta_0;$$
        thus, setting $\vartheta := \max\{1,\delta\}\in[1,\vartheta_0)$, by
        \eqref{eq.sumkexm2} we have
        $1/k\in L^{\vartheta}(I)$. 
        On the other hand, since $\alpha,\beta$ are constant,
        one has 
        $$\alpha,\beta\in W^{1,\vartheta}(I);$$
        as a consequence, according to Remark \ref{rem.assumptionH2Lp}, 
        it suffices to demonstrate that as\-sump\-tion (H6) holds
        in the weaker form (H6)' (with $\vartheta = \max\{1,\delta\}$).
        
        Let then $R > 0$ be fixed
        and let $\gamma$ be a non-negative function belonging
        to the space $L^\vartheta(I)$. Reminding that $f(t,z) = z$, we have
        the following computation
        $$\begin{array}{c}
         |f(t,z)\,\rho(t, y(t))| = |z|\cdot|y(t)|^\delta\leq
	     R\cdot\gamma(t)^\delta =: h_{R,\gamma}(t) \\[0.25cm]
        \text{for a.e.\,$t\in I$, every $z\in\R$ with $|z|\leq R$} \\[0.05cm]
        \text{and every
        $y\in L^1(I)$ such that $|y(s)|\leq \gamma(s)$ for a.e.\,$s\in I$}.
        \end{array}
        $$
        From this, since $h_{R,\gamma}\in L^1(I)$
        (remind that, by definition,
        $\delta\leq\vartheta$), we immediately conclude that
        estimate \eqref{eq.growthfrhoassWEAK} is satisfied. \medskip
        
        \textsc{Assumption (H7).} Let $R > 0$ be arbitrarily fixed. 
        Since $k\in C(I,\R)$
        (and since $f(t,z) = z$), we have the following computation
        \begin{align*}
         |f(t,z)\,\rho(t, y)| & = |z|\cdot|y|^\delta
         \leq \frac{R}{k(t)^\delta}\cdot|k(t)y|^{\delta} \\[0.2cm]
         & \big(\text{by \eqref{eq.choiceralphaexm2}, setting
         $q := \frac{\vartheta_0}{p-1} > 1$}\big) \\[0.2cm]
         & \leq |k(t)y|^{p-1}\Big(\frac{R}{k(t)^{\delta}}\cdot|k(t)y|^{\frac{q-1}{q}}\Big)
         \\[0.2cm]
         & = \Phi_p\big(|k(t)y|\big)\cdot\Big
         (\frac{R}{k(t)^{\delta+1/q-1}}\Big)\,|y|^{\frac{q-1}{q}}
        \end{align*}
        holding true for a.e.\,$t\in I$, every $z\in[-R,R]$
        and every $y\in\R$ with $|k(t)y|\geq 1$. As a consequence, 
        if we are able to demonstrate that
        \begin{equation} \label{eq.toprovesummabilitymu}
         t\mapsto \frac{R}{k(t)^{\delta+1/q-1}}\in L^q(I),
         \end{equation}
        we conclude that
        estimate \eqref{eq.Nagumocondition} is satisfied with the choice
        $$H_R = 1, \quad \psi_R(s) = s, \quad
        l_R(t) \equiv 0,\quad
        \mu_R(t) = \frac{R}{k(t)^{\delta+1/q-1}}, \quad q = \frac{\vartheta_0}{p-1}.$$
        In its turn, the needed 
        \eqref{eq.toprovesummabilitymu} follows from \eqref{eq.sumkexm2} and
        from the fact that
        $$q(\delta+1/q-1)
        = \frac{\vartheta_0}{p-1}\bigg(\delta+\frac{p-1}{\vartheta_0}-1\bigg) 
        < \frac{\vartheta_0}{p-1}\cdot(p-1) = \vartheta_0.$$
        Since all the assumptions of Theorem \ref{thm.mainPAPER} are fulfilled,
        we can conclude that there exists (at least) one solution
        $x_0\in W^{1,1}(I)$ of problem \eqref{eq.PBexm2},
        further satisfying
        $$1\leq x_0(t)\leq 2\qquad\text{for every $t\in I$}.$$
        Moreover, from \eqref{eq.sumkexm1} we deduce that $x_0\in W^{1,\vartheta}(I)$
        for all $\vartheta\in [1,\vartheta_0)$.  
      \end{example}
      \begin{example} \label{exm.ultimo}
       Let $d_1,d_2\in (0,\infty)$ be arbitrarily fixed,
       and let $I = [-1,1]$. De\-no\-ting by
       $\chi_A$ the indicator function of a set $A\subseteq\R$, we define
       \begin{equation} \label{eq.defikexm3}
        \kappa(t) := d_1\cdot\chi_{[-1,0]}(t)+d_2\cdot\chi_{[0,1]}(t).
       \end{equation}
       We then consider the following BPV:
       \begin{equation} \label{eq.PBexm3}
        \begin{cases}
         \Big(\big(\kappa(t)\,x'(t)\big)^3\Big)' + 
         \big(\max\limits_{s\in[-1,t]}x(s)\big)\cdot\log\big(1+|\sqrt[3]{t}\,x'(t)|^2\big) = 0
         & \text{a.e.\,on $I$}, \\[0.25cm]
         \,\,x(-1) = 0,\,\,
         x(1) = 1.
        \end{cases}
       \end{equation}
       Problem \eqref{eq.PBexm3} takes the form
       \eqref{eq.BVPmainSec2}, with \medskip
       
 		$(\ast)$\,\,$k:I\to\R, \quad k(t) := \kappa(t) = 
 		d_1\cdot\chi_{[-1,0]}(t)+d_2\cdot\chi_{[0,1]}(t)$; 
 		\vspace*{0.1cm}
 		
        $(\ast)$\,\,$\Phi:\R\to\R, \quad \Phi(z) := z^3$; \vspace*{0.1cm}
        
        $(\ast)$\,\,$f:I\times\R\to\R,\quad f(t,z) := z$; \vspace*{0.1cm}
        
        $(\ast)$\,\,$\rho:I\times\R\to\R, \quad \rho(t, y) := 
        \log\big(1+|\sqrt[3]{t}\,y|^2\big)$; \vspace*{0.1cm}
        
        $(\ast)$\,\,$\GG_x(t) := \max_{s\in[-1,t]}x(s)$ (for $x\in W^{1,1}(I)$); \vspace*{0.1cm}
        
        $(\ast)$\,\,$\HH_{-1}[x] := 0$
        and $\HH_1[x] := 1$ (for $x\in W^{1,1}(I)$). \medskip
        
        \noindent We aim to show that \emph{all} the assumptions
        of Theorem \ref{thm.mainPAPER} are satisfied in this case,
        so that problem \eqref{eq.PBexm3}
        possesses (at least) one solution $x_0\in W^{1,1}(I)$. 
        We explicitly point out that, in view of the boundary conditions,
		$x_0$ \emph{cannot be constant}.\medskip
        
        To begin with, we observe that assumption (H1)
        is tri\-vial\-ly satisfied, since
        \begin{equation} \label{eq.sumkexm3}
         \text{$1/k\in L^\vartheta(I)$ for all $1\leq\vartheta\leq\infty$}.
        \end{equation}
        As regards assumptions (H2)-(H3), we first notice that $\GG$ is a well-defined
        operator mapping $W^{1,1}(I)$ into $L^\infty(I)$; 
        as a consequence, since we have
        $$\|\GG_x-\GG_y\|_{L^\infty(I)}
        \leq \|x-y\|_{L^\infty(I)}\leq C\|x-y\|_{W^{1,1}(I)},$$        
        we immediately derive that
        $\GG$ is continuous from $W^{1,1}(I)$ into 
        $L^\infty(I)$ (with the usual norms).
        Furthermore, if $r > 0$ is any fixed positive number, we have 
        $$\|\GG_x\|_{L^\infty(I)}\leq \|x\|_{L^\infty(I)}
        \leq r\quad\text{for all $x\in W^{1,1}$ with $\|x\|_{L^\infty}\leq r$},$$
        and thus \eqref{eq.assumptionGGbounded} is satisfied with $\eta_r := r$.
        Finally, since $\GG$ is monotone increasing
        with respect to the point-wise order
        and since $f(t,z) = z$,
        by arguing as in Example \ref{exm.plaplacian} we derive
        that \eqref{eq.assumptionGGmonotone}
        holds with $\kappa = 0$. \vspace*{0.15cm}
        
        As regards assumption (H4), since $\HH_{-1}$
        and $\HH_1$ are \emph{constant}, 
        it is straight\-for\-ward to recognize that
        these operators are continuous (from $W^{1,1}(I)$ to $\R$)
        and monotone increasing
        (w.r.t.\,to the point-wise order). \medskip
        
        \noindent We now turn to prove the validity of assumptions
        (H5)-to-(H7). \medskip
        
        \textsc{Assumption (H5).} We claim that the constant functions
        $$\alpha(t) := 0 \qquad\text{and}\qquad\beta(t) := 1$$
        are, respectively, a lower and an upper solution
        of problem \eqref{eq.PBexm1}. 
        
        In fact, 
        since $\rho(t,0) = 0$ for all $t\in I$, we
        know from Remark \ref{rem.existencealphabeta}
        that $\alpha$ and $\beta$ are both lower and upper
        solutions of the \emph{differential equation}
	    $$\Big(\big(\kappa(t)\,x'(t)\big)^3\Big)' + 
         \big(\max\limits_{s\in[-1,t]}x(s)\big)\cdot\log\big(1+|\sqrt[3]{t}\,x'(t)|^2\big) = 0;$$
		moreover, since $\HH_{-1} \equiv 0$ and $\HH_1\equiv 1$,
        we immediately derive that $\alpha$ is a lower solution
        and $\beta$ is an upper solution
        of
        \emph{problem} \eqref{eq.PBexm3}. \medskip
        
        \textsc{Assumption (H6).} 
        We first observe that, on account of \eqref{eq.sumkexm3}, we have
        (in particular)
        $1/k\in L^2(I)$;
        moreover, since $\alpha,\beta$ are constant,
        one also has 
        $$\alpha,\beta\in W^{1,2}(I);$$
        As a consequence, according to Remark \ref{rem.assumptionH2Lp}, 
        it suffices to demonstrate that as\-sump\-tion (H6) holds
        in the weaker form (H6)' (with $\vartheta = 2$). \vspace*{0.1cm}
        
        Let then $R > 0$ be fixed
        and let $\gamma$ be a non-negative function belonging
        to $L^2(I)$. Since $\log(1+\tau)\leq\tau$ for every $\tau\geq 0$, we have
         $$\begin{array}{c}
         |f(t,z)\,\rho(t, y(t))| = |z|\cdot\log\big(1+|\sqrt[3]{t}\,y(t)|^2\big)
         \leq R\,|\sqrt[3]{t}\,y(t)|^2\leq R\,\gamma(t)^2 =: h_{R,\gamma}(t) \\[0.25cm]
        \text{for a.e.\,$t\in I=[-1,1]$, every $z\in\R$ with $|z|\leq R$} \\[0.05cm]
        \text{and every
        $y\in L^1(I)$ such that $|y(s)|\leq \gamma(s)$ for a.e.\,$s\in I$}.
        \end{array}
        $$
        As a consequence, since we have
        $h_{R,\gamma} = R\gamma^2\in L^1(I)$ (as $\gamma\in L^2(I)$), we im\-me\-dia\-te\-ly conclude that
        estimate \eqref{eq.growthfrhoass} is satisfied. \medskip
        
        \textsc{Assumption (H7).} Let $R > 0$ be arbitrarily fixed. 
        Using once again the fact that $\log(1+\tau)\leq\tau$ for all $\tau\geq 0$,
        and since $z^3\geq z^2$ if $z\geq 1$, we get the estimate
        \begin{align*}
         |f(t,z)\,\rho(t, y)| & = |z|\cdot\log\big(1+|\sqrt[3]{t}\,y|^2\big)
         \leq R\,|\sqrt[3]{t}\,y|^2\leq R\,|y|^2 \\[0.2cm]
         & (\text{setting $d := \min\{d_1,d_2\} > 0$}) \\[0.1cm]
         & \leq \frac{R}{d^2}\cdot|k(t)y|^2 \leq \frac{R}{d^2}\cdot\Phi\big(|k(t)y|\big),
        \end{align*}
        holding true for a.e.\,$t\in I$, every $z\in[-R,R]$
        and every $y\in\R$ with $|k(t)\,y|\geq 1$. 
        As a consequence, we conclude that
        estimate \eqref{eq.Nagumocondition} is satisfied with the choice
        $$H_R = 1, \quad \psi_R(s) = s, \quad
        l_R(t) := \frac{R}{d^2},\quad
        \mu_R \equiv 0.$$
        Since all the assumptions of Theorem \ref{thm.mainPAPER} are fulfilled,
        we can conclude that there exists (at least) one solution
        $x_0\in W^{1,1}(I)$ of problem \eqref{eq.PBexm3},
        further satisfying
        $$0\leq x_0(t)\leq 1\qquad\text{for every $t\in I$}.$$
        Moreover, from \eqref{eq.sumkexm3} we deduce that $x_0\in W^{1,\vartheta}(I)$
        for all $\vartheta\in [1,\infty]$. In particular, $x_0$
        is Lipschitz-continuous on $I$, but not of class $C^1$ (if $d_1\neq d_2$).
      \end{example}
 \appendix
 \section{Appendix: Continuity of truncating operators} \label{app.continuityTT}
 In this Appendix we prove in detail some properties of 
 the truncating operator. Despite these results are probably very classical,
 we were not be able to locate a precise
 reference in the literature; thus, we present here
 a complete demonstration for the sake of completeness. \bigskip
 
 To begin with, we fix a pair of functions $\omega,\zeta\in L^1(I)$ satisfying the ordering relation
 $\omega(t)\leq \zeta(t)$ a.e. in $I$, and we introduce
 the truncating operator
 $$\TT^{\,\omega,\zeta}:
 L^1(I)\to L^1(I), \qquad \TT^{\,\omega,\zeta}_x(t)=
 \min\big\{\omega(t), \max\{x(t),\zeta(t)\}\big\}.$$        
 We then prove the following result. 
 \begin{lemma} \label{lem.continuityTT}
  For every $x,y\in L^1(I)$, one has
  \begin{equation} \label{eq.TTcontrae}
   \big|\TT^{\,\omega,\zeta}_x(t)-\TT^{\,\omega,\zeta}_y(t)\big|\leq |x(t)-y(t)|.
   \end{equation}
   Moreover, if we further assume that $\omega,\zeta\in W^{1,1}(I)$, we have
   \begin{itemize}
   \item[\emph{(i)}] 
   $\TT^{\,\omega,\zeta}\big(W^{1,1}(I)\big)\subseteq W^{1,1}(I)$.
   \item[\emph{(ii)}] $\TT^{\,\omega,\zeta}$ is continuous from $W^{1,1}(I)$ into itself 
   \emph{(}with respect to the usual norm\emph{)}.
  \end{itemize}
 \end{lemma}
     \begin{proof}
     We limit ourselves to prove only assertion (ii), since
     \eqref{eq.TTcontrae} 
     is trivial and (i) is an immediate consequence of \eqref{eq.TTcontrae} and 
     the well-known characterization of $W^{1,1}(I)$ in terms of absolutely continuous functions
     (see, e.g., \cite{Brezis}). \medskip
     
     \noindent First of all we observe that, if we introduce the operators
	 \begin{equation} \label{eq.defiOpMaxMin}
	  \begin{split}
	  & \MM: W^{1,1}(I)\to W^{1,1}(I), \qquad
	  \MM_x(t) := \max\big\{\omega(t),x(t)\big\}, \\[0.2cm] 
	  &  \mathrm{m}:W^{1,1}(I)\to W^{1,1}(I), \qquad
	   \mathrm{m}_x(t) := \min\big\{\zeta(t),x(t)\big\},
	  \end{split}
	 \end{equation}
     the operator $\TT^{\,\omega,\zeta}$ is the composition
     between $\mathrm{m}$ and $\MM$, that is,
     $$\TT^{\,\omega,\zeta}_x 
     = \big(\MM\circ \mathrm{m}\big)(x) \qquad\text{for all $x\in W^{1,1}(I)$}.$$
     As a consequence, to prove the lemma it suffices to show that
     \emph{both} $\MM$ and $\mathrm{m}$ are continuous on $W^{1,1}(I)$.
	 Here we limit ourselves to demonstrate this fact only for the operator
	 $\MM$, 
	 since the case of $\mathrm{m}$ goes along the same lines. \medskip
	 
	 Let then $x_0\in W^{1,1}(I)$ be fixed, and let
	 $\{x_n\}_n\subseteq W^{1,1}(I)$ be a sequence
	 con\-ver\-ging to $x_0$ as $n\to\infty$ in $W^{1,1}(I)$.
     Moreover, let $\{y_k := x_{n_k}\}_k$ be an arbitrary sub-sequence
     of $\{x_n\}_n$. To prove the continuity of
     $\MM$ we show that, by choosing a further sub-sequence
     if necessary,
     one has 
     \begin{equation} \label{eq.toproveconTTsubseq}
      \lim_{n\to\infty}\MM_{y_k} = x_0 \qquad\text{in $W^{1,1}(I)$}.
     \end{equation}
     To ease the readability, we split the demonstration
     of \eqref{eq.toproveconTTsubseq} into some steps.
     \medskip
	 
	 \textsc{Step I.} In this step we show that
     \begin{equation} \label{eq.toproveTTyktoTTx0L1}
      \lim_{k\to\infty}\|\MM_{y_k}-\MM_{x_0}\|_{L^1(I)} = 0.
      \end{equation}
      To this end, we first notice that, since
     $y_k\to x_0$ in $W^{1,1}(I)$ as $k\to \infty$, we also have
     that $y_k$ converges \emph{uniformly on $I$}
     to $x_0$ as $k\to\infty$ (see, e.g.,
     \cite[Theorem 8.8]{Brezis}); as a consequence, since
     one can easily recognize that
     $$\|\MM_{y_k}-\MM_{x_0}\|_{L^\infty(I)}\leq \|y_k-x_0\|_{L^\infty(I)},$$
     we deduce that
     $\MM_{y_k}\to \MM_{x_0}$ uniformly on $I$ as $k\to \infty$,
     and \eqref{eq.toproveTTyktoTTx0L1} follows. \medskip
     
     \textsc{Step II.} In this step we show that, up to a sub-sequence,
     one has
     \begin{equation} \label{eq.toproveTTukprimeTTx0primeae}
      \lim_{k\to \infty}\MM_{y_k}'(t) = \MM_{x_0}'(t) \qquad\text{a.e.\,on $I$}.
     \end{equation}
     To this end, we first fix a couple of notation
     which shall be useful in the sequel. Given any point
     $t_0 \in (a,b)$ and any $\rho > 0$, we set
     $$I(t_0,\rho) := [t_0-\rho,t_0+\rho];$$
     moreover, given any function $\xi\in W^{1,1}(I)$, we define
     \begin{equation} \label{eq.defiNomega}
     \mathcal{N}_\xi := \big\{t\in(a,b):\,\text{$\xi$
     is not differentiable at $t$}\big\}.
     \end{equation}
     Notice that, since $\xi\in W^{1,1}(I)$, the set $\mathcal{N}_\xi$
     has zero Lebesgue measure. \vspace*{0.1cm}
     
     We now start with the proof of \eqref{eq.toproveTTukprimeTTx0primeae}.
     First of all, since $y_k\to x_0$
     in $W^{1,1}(I)$ as $k\to\infty$, we clearly have that
     $y_k'\to x_0'$ in $L^1(I)$ (as $k\to\infty$); as a consequence,
     it is possible to find a non-negative function $g\in L^1(I)$
     and a set $\mathcal{Z}\subseteq I$, with vanishing Lebesgue measure,
     such that (up to a sub-sequence) \medskip
     
     (i)\,\,$
     y_k'(t)\to x'_0(t)$ as $k\to\infty$ for every $t\in I\setminus \mathcal{Z}$; \medskip
     
     (ii)\,\,$|y_k'(t)|\leq g(t)$ for every $k\in\N$ and every $t\in I\setminus\mathcal{Z}$. \medskip
     
     \noindent With reference to \eqref{eq.defiNomega}, we
      consider the following set:
     \begin{equation} \label{eq.defiSetNN}
     \mathcal{N} :=
     \bigcup_{k\in\N}\mathcal{N}_{y_k}\cup
     \bigcup_{k\in\N}\mathcal{N}_{\MM_{y_k}}
     \cup \mathcal{N}_{x_0}\cup \mathcal{N}_{\MM_{x_0}}
     \cup \mathcal{N}_\omega\cup\mathcal{Z}.
     \end{equation}
	 Since it is a countable union of sets
	 with zero Lebesgue measure, the set
	 $\mathcal{N}$ has zero Lebesgue measure as well;
	 thus, to prove
	 \eqref{eq.toproveTTukprimeTTx0primeae} it suffices to show that
	 \begin{equation} \label{eq.toproveTTykprimeTTx0primeoutN}
	  \lim_{k\to\infty}\MM_{y_k}'(t) = \MM_{x_0}'(t) \qquad\text{for all
	  $t\in (a,b)\setminus\mathcal{N}$}.
	 \end{equation}
	 Let then $\theta_0\in(a,b)\setminus\mathcal{N}$ be arbitrary but fixed.
	 We demonstrate the claimed \eqref{eq.toproveTTykprimeTTx0primeoutN}
	 by analyzing separately the following three possibilities. \medskip
	 
	 (1)\,\,$x_0(\theta_0)<\omega(\theta_0)$. In this case,
	 we let $\rho > 0$ be so small that 
	 \begin{equation} \label{eq.choiceIthetarhocase2}
	  \text{$I(\theta_0,\rho)\subseteq(a,b)$
	 and 
	 $x_0 < \omega$ on $I(\theta_0,\rho)$}.
	 \end{equation}
	 Owing to the very definition
	 of $\MM$ in \eqref{eq.defiOpMaxMin}, we get
	 $\MM_{x_0}\equiv\omega$ on $I(\theta_0,\rho)$; moreover,
	 since $\omega$ is differentiable in
	 $\theta_0\notin\mathcal{N}_{\omega}$, we have
	 \begin{equation} \label{eq.TTx0primeequivalbetaprimecase2}
	  \MM_{x_0}'(\theta_0) = \omega'(\theta_0).
	  \end{equation}
	 Now, since we know from Step I that $y_k$ converges \emph{uniformly on $I$} to $x_0$
	 as $k\to \infty$, by \eqref{eq.choiceIthetarhocase2} we can find a natural
	 number $\kappa_0$ such that
	 $$y_k(t) < \omega(t) \qquad \text{for $t\in I(\theta_0,\rho)$ 
	 and every $k \geq \kappa_0$;}$$
	 thus, again by definition of $\MM$ we deduce that
	 $\MM_{y_k}\equiv \omega$ on $I(\theta_0,\rho)$
	 for all $k\geq \kappa_0$.
     In particular, $\omega$ being differentiable
	 at $\theta_0$
	 we have
	 \begin{equation} \label{eq.TTykprimeequivbetaprimecase2}
	  \MM_{y_k}'(\theta_0) = \omega'(\theta_0) \qquad\text{for
	 every $k\geq\kappa_0$}.
	 \end{equation}
	 Gathering together 
	 \eqref{eq.TTx0primeequivalbetaprimecase2} 
	 and \eqref{eq.TTykprimeequivbetaprimecase2}, we then
	 obtain \eqref{eq.toproveTTykprimeTTx0primeoutN} in this case. \medskip
	 
	 (2)\,\,$x_0(\theta_0) > \omega(\theta_0)$. In this case,
	 we let $\rho > 0$ be so small that 
	 \begin{equation} \label{eq.choiceIthetarhocase3}
	  \text{$I(\theta_0,\rho)\subseteq(a,b)$
	 and 
	 $x_0 > \omega$ on $I(\theta_0,\rho)$}.
	 \end{equation}
	 Owing to the very definition
	 of $\MM$ in \eqref{eq.defiOpMaxMin}, we get
	 $\MM_{x_0}\equiv x_0$ on $I(\theta_0,\rho)$; moreover,
	 since $x_0$ is differentiable in
	 $\theta_0\notin\mathcal{N}_{x_0}$, we have
	 \begin{equation} \label{eq.TTx0primeequivalbetaprimecase3}
	  \MM_{x_0}'(\theta_0) = x_0'(\theta_0).
	  \end{equation}
	 Now, using \eqref{eq.choiceIthetarhocase3} and arguing again as in
	 (1), we can find $\kappa_0\in\N$ such that
	 $$\text{$\MM_{y_k}\equiv y_k$ on $I(\theta_0,\rho)$
	 for all $k\geq \kappa_0$;}$$
     in particular, $y_k$ being differentiable
	 at $\theta_0$ for all $k\in\N$,
	 we have
	 \begin{equation} \label{eq.TTykprimeequivbetaprimecase3}
	  \MM_{y_k}'(\theta_0) = y_k'(\theta_0) \qquad\text{for
	 every $k\geq\kappa_0$}.
	 \end{equation}
	 Since $\theta_0\notin\mathcal{Z}$
	 and
	 since $y_k'\to x_0'$ as $k\to\infty$ on $I\setminus\mathcal{Z}$, by combining
	 \eqref{eq.TTx0primeequivalbetaprimecase3} with 
	 \eqref{eq.TTykprimeequivbetaprimecase3}
	 we readily conclude that
	 \eqref{eq.toproveTTykprimeTTx0primeoutN}
	 holds
	 also in this case. \medskip
	 
	 (3)\,\,$x_0(\theta_0) = \omega(\theta_0)$.
     First of all, since for every $k\in\N$
	 the function $\MM_{y_k}$ is differentiable
	 at $\theta_0$ (as $\theta_0\notin\mathcal{N}$, see \eqref{eq.defiSetNN}),
	 by the very definition of $\MM$ we have
	 \begin{equation*}
	  \MM_{y_k}'(\theta_0)\in\big\{\omega'(\theta_0),\,y_k'(\theta_0)\big\}.
	 \end{equation*}
	 As a consequence, since we know that $y_k'(\theta_0)\to x_0'(\theta_0)$ as
	 $k\to\infty$ (as $\theta_0\notin\mathcal{Z}$, see (i)
	 at the beginning of this step), to prove \eqref{eq.toproveTTykprimeTTx0primeoutN}
     it suffices to show that
	 \begin{equation} \label{eq.expressionderivativeTTx0case4}
	  \MM_{x_0}'(\theta_0) = x_0'(\theta_0) = \omega'(\theta_0).
	 \end{equation}
     To establish \eqref{eq.expressionderivativeTTx0case4} 
     we need to consider three different sub-cases. \medskip
     
     $\mathrm{(3)_1}$\,\,$\theta_0\notin\de\{x_0 > \omega\}$.
	 In this case, since $\mathcal{O} := \{x_0 > \omega\}$ is open
	 and $\theta_0\notin\mathcal{O}$,
     there exists $\rho > 0$ such that
	 $x_0 \leq \omega$ on $I(\theta_0,\rho)$; thus,
	 by \eqref{eq.defiOpMaxMin} we have
	 $$\text{$\MM_{x_0}\equiv \omega$ on $I(\theta_0,\rho)$}.$$
	 Since $\omega$ is differentiable at $\theta_0$
	 (as $\theta_0\notin\mathcal{N}_\omega$), we then obtain
	 \begin{equation} \label{eq.derTTx0case41}
	  \MM_{x_0}'(\theta_0) = \omega'(\theta_0).
	  \end{equation}
     On the other hand, since $x_0-\omega\leq 0$ on
     $I(\theta_0,\rho)$ and $x_0(\theta_0) = \omega(\theta_0)$, we see that
     $\theta_0$ is an interior maximum point for $x_0-\omega$
     on $I(\theta_0,\rho)$; this function being differentiable
     at $\theta_0$, we then conclude that 
     \begin{equation}  \label{eq.derx0case41}
      x_0'(\theta_0) = \omega'(\theta_0).
      \end{equation}
	 Gathering together \eqref{eq.derTTx0case41} and 
	 \eqref{eq.derx0case41}, we obtain
     \eqref{eq.expressionderivativeTTx0case4} in this case.
	 \medskip
	 
	 $\mathrm{(3)_2}$\,\,$\theta_0\notin\de\{x_0 < \omega\}$. 
	 In this case, since $\mathcal{O} := \{x_0 < \omega\}$ is open
	 and $\theta_0\notin\mathcal{O}$,
     there exists $\rho\in(0,\rho_0)$ such that
	 $x_0 \geq \omega$ on $I(\theta_0,\rho)$; thus,
	 by \eqref{eq.defiOpMaxMin} we have
	 $$\text{$\MM_{x_0}\equiv x_0$ on $I(\theta_0,\rho_0)$}.$$
	 From this, by arguing exactly as in case $\mathrm{(3)_1}$,
	 we obtain
     \eqref{eq.expressionderivativeTTx0case4}. \medskip
     
     $\mathrm{(3)_3}$\,\,$\theta_0\in\de\{x_0<\omega\}\cap\de\{x_0 > \omega\}$.
	 In this last case, both the open sets
	 \begin{align*}
	 & \mathcal{O}^+ = \big\{x_0>\omega\big\}
	 \qquad \text{and} \qquad
	  \mathcal{O}^- = \big\{x_0<\omega\big\}
	 \end{align*}
	 are non-empty and $\theta_0\in\de(\mathcal{O}^+)\cap
	 \de(\mathcal{O}^-)$; thus, by crucially exploiting
	 the fact that the functions $\MM_{x_0}$ and
	 $x_0$ are dif\-fe\-ren\-tia\-ble at
	 $\theta_0$, we can write
	 \begin{align*}
	  \MM_{x_0}'(\theta_0) & = \lim_{\begin{subarray}{c}
	   t\to\theta_0 \\
	   t\in I
	  \end{subarray}}
	  \frac{\MM_{x_0}(t)-\MM_{x_0}(\theta_0)}{t-\theta_0} \\[0.1cm]
	  &\big(\text{since $\MM_{x_0}(\theta_0) = x_0(\theta_0) = \omega(\theta_0)$, see
	  \eqref{eq.defiOpMaxMin}}\big) \\[0.1cm]
	  & = \lim_{\begin{subarray}{c}
	   t\to\theta_0 \\
	   t\in \mathcal{O}^+
	  \end{subarray}}
	  \frac{\MM_{x_0}(t)-x_0(\theta_0)}{t-\theta_0} \\[0.1cm] 
	  & \big(\text{since $x_0 > \omega$ on $\mathcal{O}^+$}\big) \\[0.1cm]
	  & = \lim_{\begin{subarray}{c}
	   t\to\theta_0 \\
	   t\in \mathcal{O}^+
	  \end{subarray}
	  }\frac{x_0(t)-x_0(\theta_0)}{t-\theta_0} = x_0'(\theta_0).
	 \end{align*}
	 On the other hand, using the fact that
	 $\omega$ is differentiable at $\theta_0$, we also have
	 \begin{align*}
	  \MM_{x_0}'(\theta_0) & = \lim_{\begin{subarray}{c}
	   t\to\theta_0 \\
	   t\in I
	  \end{subarray}}
	  \frac{\MM_{x_0}(t)-\MM_{x_0}(\theta_0)}{t-\theta_0} \\[0.1cm]
	  & = \lim_{\begin{subarray}{c}
	   t\to\theta_0 \\
	   t\in \mathcal{O}^-
	  \end{subarray}}
	  \frac{\MM_{x_0}(t)-\omega(\theta_0)}{t-\theta_0} \\[0.1cm] 
	  & \big(\text{since $x_0 < \omega$ on $\mathcal{O}^-$}\big) \\[0.1cm]
	  & = \lim_{\begin{subarray}{c}
	   t\to\theta_0 \\
	   t\in \mathcal{O}^-
	  \end{subarray}
	  }\frac{\omega(t)-\omega(\theta_0)}{t-\theta_0} = \omega'(\theta_0).
	 \end{align*}
	 Gathering together these two facts, we conclude that
	 \begin{equation} \label{eq.derTTx0derx0deralphacase4}
	  \MM_{x_0}'(\theta_0) = x_0'(\theta_0) = \omega'(\theta_0),
	 \end{equation}
	 which is exactly the desired \eqref{eq.expressionderivativeTTx0case4}. \medskip
	 
	 \textsc{Step III.} In this step we prove that, up to a sub-sequence,
	 one has
	 \begin{equation} \label{eq.toproveStepIII}
	  \lim_{k\to\infty}\|\MM_{y_k}' - \MM_{x_0}'\|_{L^1(I)} = 0.
	 \end{equation}
	 To begin with, by exploiting
	 the results in Step II, we know that there exists
	 a set $\mathcal{N}\subseteq(a,b)$, with zero Lebesgue measure,
	 such that (up to a sub-sequence) \medskip
	 
	 (a)\,\,$y_k'\to x_0'$ point-wise on $I\setminus\mathcal{N}$; \medskip
	 
	 (b)\,\,$|y_k'|\leq g$ on $I\setminus\mathcal{N}$ for a suitable
	 function $g\in L^1(I)$; \medskip
	 
	 (c)\,\,$\MM_{y_k}'\to \MM_{x_0}'$ point-wise on $I\setminus\mathcal{N}$. \medskip
	 
	 \noindent In particular, since for every $k\in\N$ we have
	 $$\MM_{y_k}'\in \big\{\omega'(t),\,y_k'(t)\big\} \qquad\text{a.e.\,on $I$},$$
	 from (b) we obtain the following estimate
	 \begin{equation} \label{eq.estimcrucialTTyk}
	  |\MM_{y_k}'(t)| \leq |\omega'(t)|+g(t) =: \xi(t), 
	  \qquad\text{for a.e.\,$t\in I$}.
	 \end{equation}
	 By combining 
	 \eqref{eq.estimcrucialTTyk} with (c) we can perform a stan\-dard do\-mi\-na\-ted\--convergence
	 argument, proving the claimed \eqref{eq.toproveStepIII}. \medskip
	 
	 \textsc{Step IV.} In this last step we complete the demonstration of the lemma.
	 By combining \eqref{eq.toproveTTyktoTTx0L1} in Step I
	 with \eqref{eq.toproveStepIII} in Step III, we straightforwardly get
	 \begin{align*}
	  & \lim_{k\to\infty}\|\MM_{y_k}-\MM_{x_0}\|_{W^{1,1}(I)} \\
	 & \qquad = \lim_{k\to\infty}\big(\|\MM_{y_k}-\MM_{x_0}\|_{L^1(I)}
	 + \|\MM_{y_k}'-\MM_{x_0}'\|_{L^1(I)}\big) = 0,
	 \end{align*}
	 and this is exactly our starting goal (see \eqref{eq.toproveconTTsubseq}). 
	 This ends the proof.
	 \end{proof}
     \begin{remark} \label{rem.dimostrazionegiusta}
	  As a matter of fact, in the recent paper \cite{CaMaPa2018} it is contained a
	  proof of Lemma \ref{lem.continuityTT}; however, it seems that
	  this proof is not correct. We thus take this occasion to correct the mistake
	  in \cite{CaMaPa2018} by giving a new proof of Lemma \ref{lem.continuityTT}.
	 \end{remark}

\end{document}